\documentclass[12pt,a4paper]{article}
\usepackage[T1]{fontenc}
\usepackage{lmodern}
\usepackage[utf8]{inputenc}
\usepackage{amsthm}
\usepackage{amsmath,amssymb}
\usepackage{graphics}
\usepackage{lscape,graphicx}
\usepackage{enumitem}
\usepackage{amsfonts}
\usepackage[T1]{fontenc}
\usepackage{longtable}
\usepackage{authblk}
\usepackage{lipsum}
\usepackage{epsfig}
\usepackage{scalerel}
\usepackage{todonotes}
\usepackage{float}
\usepackage{hyperref}
\usepackage{comment}
\usepackage[normalem]{ulem}
\usepackage{multicol}
\usepackage{fancyhdr}
\usepackage[english]{babel}
\usepackage{mathrsfs}
\usepackage{tocloft}
\usepackage{xcolor}
\usepackage{titletoc}
\usepackage{mathtools}
\usepackage[toc,page]{appendix}

\addto{\captionsenglish}{}
\makeatletter
\def\thm@space@setup{%
  \thm@preskip=1cm plus .5cm minus .5cm
  \thm@postskip=.5cm plus .6cm minus .5cm % or whatever, if you don't want them to be equal
}

\makeatother

\newtheorem{thm}{Theorem}
\newtheorem{fact}{Fact}
\newtheorem{lma}{Lemma}

\newtheorem{rmk}{Remark}
\newtheorem*{ex}{Examples}
\numberwithin{thm}{section}
\numberwithin{lma}{section}
\numberwithin{dfn}{section}
\numberwithin{cor}{section}
\numberwithin{rmk}{section}
\numberwithin{prop}{section}

\newcommand*{\thmref}[1]{Theorem~\ref{#1}}

\newcommand*{\lmaref}[1]{Lemma~\ref{#1}}
\newcommand*{\rmkref}[1]{Remark~\ref{#1}}

\title{A subset generalization of the Erd\H{o}s-Kac theorem over number fields with applications}

\date{}
\begin{document}

% \author{Sourabhashis Das, Wentang Kuo, Yu-Ru Liu}

% \newcommand{\Addresses}{{% additional braces for segregating \footnotesize
%   \bigskip
%   \footnotesize

%   Sourabhashis Das (Corresponding author), Department of Pure Mathematics, University of Waterloo, 200 University Avenue West, Waterloo, Ontario, Canada, N2L 3G1. \\
%   Email address: \texttt{s57das@uwaterloo.ca}

%   \medskip

%   Wentang Kuo, Department of Pure Mathematics, University of Waterloo, 200 University Avenue West, Waterloo, Ontario, Canada, N2L 3G1. \\
%   Email address: \texttt{wtkuo@uwaterloo.ca}
  
%   \medskip

%   Yu-Ru Liu, Department of Pure Mathematics, University of Waterloo, 200 University Avenue West, Waterloo, Ontario, Canada, N2L 3G1. \\
%   Email address: \texttt{yrliu@uwaterloo.ca}

% }}

\author[1]{Sourabhashis Das}
\affil[1]{Corresponding author, Department of Pure Mathematics, University of Waterloo, 200 University Avenue West, Waterloo, Ontario, Canada, N2L 3G1, \textit{s57das@uwaterloo.ca}.}
%
%\author[2]{Ertan Elma}
%\affil[2]{Department of Mathematics and Computer Science\\
%University of Lethbridge\\
%Alberta, Canada \\ \textit{ertan.elma@uleth.ca}}
%
\author[2]{Wentang Kuo}
\affil[2]{Department of Pure Mathematics, University of Waterloo, 200 University Avenue West, Waterloo, Ontario, Canada, N2L 3G1, \textit{wtkuo@uwaterloo.ca}.}
\author[3]{Yu-Ru Liu}
\affil[3]{Department of Pure Mathematics, University of Waterloo, 200 University Avenue West, Waterloo, Ontario, Canada, N2L 3G1, \textit{yrliu@uwaterloo.ca}.}

\maketitle 

% \iffalse
% \begin{center}
%  \textbf{\underline{On the number of irreducible factors with a given multiplicity in function fields}}
% \vspace{.3cm}
% %\textbf{(By Sourabh)}
% \end{center}
% \fi
%\tableofcontents
\begin{abstract}
Let $\omega(n)$ denote the number of distinct prime factors of a natural number $n$. In \cite{ErdosKac}, Erd\H{o}s and Kac established that $\omega(n)$ obeys the Gaussian distribution over natural numbers. In \cite{liu}, the third author generalized their theorem to all abelian monoids. In this work, we extend the work of \cite{liu} to any subset of the set of ideals of a number field satisfying some additional conditions. Finally, we apply this theorem to prove the Erd\H{o}s-Kac theorem over $h$-free and over $h$-full ideals of the number field.  
\end{abstract}

\section{Introduction}
For\footnotetext{\textbf{2020 Mathematics Subject Classification: 11R04, 11R44, 11R45.}}\footnotetext{\textbf{Keywords: number fields, prime ideals, the $\omega$-function, Erd\H{o}s-Kac theorem.}}\footnotetext{The research of W. Kuo and Y.-R. Liu is supported by the Natural Sciences and Engineering Research Council of Canada (NSERC) discovery grants.} $n \in \mathbb{N}$, let $\omega(n)$ count the number of distinct prime divisors of $n$. In \cite{hardyram}, Hardy and Ramanujan established that $\omega(n)$ takes the value $\log \log n$ almost all the time by proving that $\omega(n)$ has normal order $\log \log n$ over natural numbers. This result was followed up by the work of Erd\H{o}s and Kac \cite{ErdosKac} who used probabilistic means to establish that $\omega(n)$ obeys the Gaussian distribution with mean $\log \log n$ over naturals. 
% In \cite{ErdosKac}, Erd\H{o}s and Kac established a pioneering result that $\omega(n)$ obeys the Gaussian distribution over naturals. 
In particular, for any $\gamma \in \mathbb{R}$, they proved
$$\lim_{x \rightarrow \infty} \frac{1}{x}  \left| \left\{ n \leq x \ : \ n \geq 3, \ \frac{\omega(n) - \log \log n}{\sqrt{\log \log n}} \leq \gamma \right\} \right| = \Phi(\gamma),$$
where 
\begin{equation}\label{phi(a)}
    \Phi(\gamma) = \frac{1}{\sqrt{2 \pi}} \int_{-\infty}^\gamma e^{-u^2/2} \ du,
\end{equation}
and where $|S|$ denotes the cardinality of the set $S$. Following their work, various approaches to the Erd\H{o}s-Kac theorem have been pursued, for example, see \cite{pb}, \cite{lg}, \cite{gs}, \cite{hh1}, \cite{hh2}, \cite{hh3}, and \cite{jala2}. Recently, in \cite{mmp}, Murty, Murty, and Pujahari proved an all-purpose Erd\H{o}s-Kac theorem which applies to diverse settings. Moreover, the third author in \cite{liu} provided a generalization of the Erd\H{o}s-Kac theorem over any abelian monoid. We extend this work to provide another generalization of the Erd\H{o}s-Kac theorem over any subset of the set of ideals of a number field satisfying some additional conditions. 

Let $K/\mathbb{Q}$ be a number field of degree $n_K = [K: \mathbb{Q}]$ and $\mathcal{O}_K$ be its ring of integers. Let $\mathcal{P}$ be the set of prime ideals of $\mathcal{O}_K$ and $\mathcal{M}$ be the set of ideals of $\mathcal{O}_K$. Let $N: \mathcal{M} \rightarrow \mathbb{N}$ be the standard norm map, i.e., $\mathfrak{m} \mapsto N(\mathfrak{m}) := |\mathcal{O}_K/ \mathfrak{m}|$.
%and choose $X = \mathbb{Q}_{\geq 0}$, i.e, the set of positive rational numbers.

% Let $K$ be a number field with $\mathbb{P}$ be a set of elements with a map
% $$N: \mathbb{P} \rightarrow \mathbb{N}\backslash \{1\}, \quad \wp \mapsto N(\wp).$$
% We call this map the norm map. Let $\mathcal{M}$ be a free abelian monoid generated by elements of $\mathbb{P}$. For each $m \in M$, we write
% $$m = \sum_{p \in  \mathbb{P}} n_p(m) p,$$
% with $n_p(m) \in \mathbb{N} \cup \{ 0 \}$ and $n_p(m) =0$ for all but finitely many $p$. We extend the map $N$ on $M$ as the following:
% \begin{align*}
%     N : M & \rightarrow \mathbb{N} \\
%     m = \sum_{p \in \mathbb{P}} n_p(m) p & \longmapsto N(m) := \prod_{p \in \mathbb{P}} N(p)^{n_p(m)}. 
% \end{align*}
% Thus, $N$ can be extended to a monoid homomorphism from $(M,+)$ to $(\mathbb{N},\cdot)$. Let $X$ be a countable subset of $\mathbb{Q}$ that contains the image $\text{Im}(N(M))$ with an extra condition: if $x_1,x_2 \in X$, the fraction $x_1/x_2$ belongs to $X$, too. Without loss of generality, we assume $X = \mathbb{Q}$ or $X = \{ q^z \ : \ z \in \mathbb{Z} \}$ for some $q \in \mathbb{N}$. For interested readers, the details behind $X$ being limited to these two choices are presented in \cite[Theorem 2]{liu}.

Let $\mathcal{S}$ be a subset of infinitely many elements of $\mathcal{M}$. For $x >1$, we define
$$\mathcal{S}(x) = \{ \mathfrak{m} \in \mathcal{S} \ : \ N(\mathfrak{m}) \leq x \}.$$
We assume that $\mathcal{S}$ satisfies the following condition:
\begin{equation}\label{Scondition}
    |\mathcal{S}(x^{1/2})| = o(|\mathcal{S}(x)|), \quad \text{for all } x > 1.
\end{equation}
% \todo[inline]{The following change has been implemented throughout the article when required: we would like $\lambda_\wp$ below to be a constant independent of $x$, since later in Lemma 3.5, we need $Y_\wp = X_\wp - \lambda_\wp$ to be a random variable. The cases we apply the main theorem ($h$-free and $h$-full) also only require $\lambda_\wp$ to be constant. Thus, we implement the change here (from $\lambda_\wp(x)$ to $\lambda_\wp$). Note that a constant is a random variable taking the value of the constant everywhere, and the sum of two random variables is a random variable. This makes $Y_\wp$ a random variable.}

For each $\wp \in \mathcal{P}$, we write
\begin{equation}\label{eq_lambdap}
    \frac{1}{|\mathcal{S}(x)|} \left| \{ \mathfrak{m} \in \mathcal{S}(x) \ : \ \wp | \mathfrak{m} \} \right| = \lambda_\wp+ e_\wp(x),
\end{equation}
where $\lambda_\wp$ denotes the main term that is independent of $x$ and $e_\wp(x)$ is the error term. In the following, we will use $e_{\wp}$ to abbreviate $e_{\wp}(x)$. For any sequence of distinct elements $\wp_{\scaleto{1}{4pt}}, \wp_{\scaleto{2}{4pt}}, \cdots, \wp_{\scaleto{u}{4pt}} \in \mathcal{P}$, we write  
$$\frac{1}{|\mathcal{S}(x)|} \left| \{ \mathfrak{m} \in \mathcal{S}(x) \ : \ \wp_i | \mathfrak{m} \text{ for all } i = 1, \cdots, u \} \right|= \lambda_{\wp_{\scaleto{1}{3pt}}}\cdots\lambda_{\wp_{\scaleto{u}{3pt}}} + e_{\wp_{\scaleto{1}{3pt}}\cdots \wp_{\scaleto{u}{3pt}}}(x).$$
In the following, we will use $e_{\wp_{\scaleto{1}{3pt}} \cdots \wp_{\scaleto{u}{3pt}}}$ to abbreviate $e_{\wp_{\scaleto{1}{3pt}} \cdots \wp_{\scaleto{u}{3pt}}}(x)$.

% \todo[inline]{For Dr.\ Kuo, your comment said that $\beta$ below depends on $x$. However, I don't think we need that. $\beta$ being any constant depending on $h$ should work. In our work, we need $\beta$ to be 1 ($h$-free case) and $1/h$ ($h$-full case). Also, the little-$o$ below should be independent of $\beta$ in our work. However, if we state it as dependent on $\beta$, it should not be a problem either as $\beta$ is a fixed value (not dependent on $x$). In particular, parts (b) and (c) are used in the proof of Lemma 3.2. Having little-$o$ depend on $\beta$ would not change the statement of the Lemma.}

Suppose there exists a $\beta$ with $0 < \beta \leq 1$ and $y = y(x) < x^\beta$ such that the following conditions are satisfied:
   \begin{enumerate}
       \item[(a)] $\left|\{ \wp \in \mathcal{P} \ : \ N(\wp) > x^\beta, \ \wp| \mathfrak{m} \} \right| = O_\beta(1)$ for each $\mathfrak{m} \in \mathcal{S}(x)$. Here, $O_{Y}$ denotes that the big-O constant depends on the variable set $Y$. 
       \item[(b)] $\sum_{y < N(\wp) \leq x^\beta} \lambda_\wp = o \left( (\log \log x)^{1/2} \right)$.
       \item[(c)] $\sum_{y < N(\wp) \leq x^\beta} |e_\wp| = o \left( (\log \log x)^{1/2} \right)$.
       \item[(d)] $\sum_{N(\wp) \leq y} \lambda_\wp = \log \log x + o \left( (\log \log x)^{1/2} \right)$.
       \item[(e)] $\sum_{N(\wp) \leq y} \lambda_\wp^2 = o \left( (\log \log x)^{1/2} \right)$.
       \item[(f)] For $r \in \mathbb{N}$, let $u$ be any integer picked from $\{ 1, 2, \cdots, r \}$. We have
       $$\sum{\vphantom{\sum}}'' |e_{\wp_{\scaleto{1}{3pt}} \cdots \wp_{\scaleto{u}{3pt}}}| = o \left( (\log \log x)^{-r/2} \right), $$
       where $\sum{\vphantom{\sum}}''$ extends over all $u$-tuples $(\wp_{\scaleto{1}{4pt}},\wp_{\scaleto{2}{4pt}},\cdots,\wp_{\scaleto{u}{4pt}})$ with $N(\wp_i) \leq y$ for all $i \in \{ 1, 2, \cdots, u\}$ and all $\wp_i'$s are distinct.
   \end{enumerate}
For $\mathfrak{m} \in \mathcal{M}$, let $\omega(\mathfrak{m})$ count the number of distinct prime ideals dividing $\mathfrak{m}$. Thus, we can write
$$\omega(\mathfrak{m}) = \sum_{\substack{\wp \in \mathcal{P} \\ \wp | \mathfrak{m}}} 1.$$
% the number of elements of $\mathcal{P}$ that generates $\mathfrak{m}$, counted without multiplicity. 
Using this definition and the above conditions, we obtain the following subset generalization of the Erd\H{o}s-Kac theorem over number fields:
\begin{thm}\label{yurugen}
Let $\mathcal{S}$ be a subset of $\mathcal{M}$. For any $x > 1$, let $\mathcal{S}(x)$ be the set of elements in $\mathcal{S}$ with norm $N(\cdot)$ less than or equal to $x$. Assume that $\mathcal{S}$ satisfies the condition \eqref{Scondition}. Suppose there exists a $\beta$ with $0 < \beta \leq 1$ and $y = y(x) < x^\beta$ such that the conditions \textnormal{(a)} to \textnormal{(f)} above hold. 
Then for $\gamma \in \mathbb{R}$, we have
$$\lim_{x \rightarrow \infty} \frac{1}{|\mathcal{S}(x)|} \left| \left\{ \mathfrak{m} \in \mathcal{S}(x) \ : \ N(\mathfrak{m}) \geq 3, \ \frac{\omega(\mathfrak{m}) - \log \log N(\mathfrak{m})}{\sqrt{\log \log N(\mathfrak{m})}} \leq \gamma\right\} \right| = \Phi(\gamma),$$
where $\Phi(\gamma)$ is defined in \eqref{phi(a)}.
\end{thm}
We list some well-studied applications of this general setting.
\begin{ex}
% Using $S = M$, we recover all the applications of the generalized Erd\H{o}s-Kac theorem mentioned in \cite[Examples 1-4]{liu}. For example, in the case of rationals, let $\mathbb{P}$ be the set of primes of $\mathbb{N}$ with the identity map $N$. 
Using $\mathcal{S} = \mathcal{M}$, we recover the applications of the generalized Erd\H{o}s-Kac theorem to the number fields mentioned in \cite[Examples 1-2]{liu}. In particular, let $\mathcal{S}=\mathcal{M}=\mathbb{N}$ with the norm $N$ being the identity map. Then by \thmref{yurugen}, we recover the classical Erd\H{o}s-Kac theorem. 
% Similarly, we obtain the generalization of Erd\H{o}s-Kac theorem in the case of other number fields.
\end{ex}
% \begin{ex}
% Let $\mathbb{P}$ be the set of primes of $\mathbb{N}$ with the identity map $N$. Let $M=\mathbb{N}$, $X= \mathbb{Q}$, and $f$ be the identity map. For a natural number $n$, let the prime factorization of $n$ be given as
% \begin{equation*}
% n = p_1^{s_1} \cdots p_r^{s_r},
% \end{equation*}
% where $p_i$'s are its distinct prime factors and $s_i$'s are their respective multiplicities. Let $h \geq 2$ be an integer. We say $n$ is $h$-free if $s_i \leq h-1$ for all $i \in \{1, \cdots, r\}$. Let $\mathcal{S}_h$ denote the set of $h$-free numbers. Moreover, we say a natural number $n$ is $h$-full if $s_i \geq h$ for all $i \in \{1, \cdots, r\}$. Let $\mathcal{N}_h$ denote the set of all $h$-full numbers. In \cite{dkl3}, using \thmref{yurugen} with $S= \mathcal{S}_h$ and $S = \mathcal{N}_h$ respectively, we prove that $\omega(n)$ satisfies the Erd\H{o}s-Kac theorem over $h$-free and $h$-full numbers repsectively.
% \end{ex}
In the following part, we discuss two other instances where \thmref{yurugen} can be applied on proper subsets of $\mathcal{M}$, in particular, the cases of $h$-free ideals and $h$-full ideals. We begin with some definitions and some distribution results concerning the $h$-free and the $h$-full ideals of a number field $K$.

For an ideal $\mathfrak{m} \in \mathcal{M}$, let the prime ideal factorization of $\mathfrak{m}$ be given as
\begin{equation}\label{factorization}
\mathfrak{m} = \wp_{\scaleto{1}{3pt}}^{s_1} \cdots \wp_r^{s_r},
\end{equation}
where $\wp_i'$s are its distinct prime ideal factors and $s_i'$s are their respective multiplicities. Here, $\omega(\mathfrak{m}) = r$. Let $h \geq 2$ be an integer. We say $\mathfrak{m}$ is $h$\textit{-free} if $s_i \leq h-1$ for all $i \in \{1, \cdots, r\}$, and we say $\mathfrak{m}$ is $h$\textit{-full} if $s_i \geq h$ for all $i \in \{1, \cdots, r\}$. Let $\mathcal{S}_h$ denote the set of $h$-free ideals and $\mathcal{N}_h$ denote the set of $h$-full ideals. 
%In the following subsection, we study the distribution of $h$-free ideals in the number field $K$.

Next, we present the distribution of $h$-free and $h$-full ideals in $K$. For a number field $K$, we recall that $\zeta_K(s)$ refers to the Dedekind zeta function of $K$:
\[
\zeta_K(s) = \sum_{\substack{\mathfrak{m} \in \mathcal{M} \\ \mathfrak{m} \neq 0}} \frac1{(N(\mathfrak{m}) )^s} 
= \prod_{\wp \in \mathcal{P}} \Big( 1- N (\wp)^{-s}\Big)^{-1}
\ \text{for } \Re (s) >1,
\]
where $\mathfrak{m}$  and $\wp$ respectively range through the non-zero ideals and the prime ideals of $\mathcal{O}_K$. Let $\kappa$ be the residue of the simple pole of $\zeta_K(s)$ at $s=1$ and be given by
\begin{equation}\label{defkappa}
    \kappa = \kappa_K = \frac{2^{r_1}(2 \pi)^{r_2} h R}{\nu \sqrt{|d_K|}},
\end{equation}
with
\begin{align*}
    r_1 & = \text{the number of real embeddings of } K,\\
    2 r_2 & = \text{the number of complex embeddings of } K,\\
    h & = \text{the class number},\\
    R & = \text{the regulator}, \\
    \nu & = \text{the number of roots of unity}, \\
    d_K & = \text{the discriminant of } K.
\end{align*}
Note that, if $K = \mathbb{Q}$, then $\kappa = 1$. 

For $x > 1$, let $I_K(x)$ count the number of ideals with norm $N(\cdot)$ less than or equal to $x$. Landau in \cite[Satz 210]{Landau} proved that
\begin{equation}\label{Landaueq}
    I_K(x) = \kappa x + O \left( x^{1 - \frac{2}{n_{\scaleto{K}{3pt}}+1}}\right),
\end{equation}
where $\kappa$ is defined in \eqref{defkappa}. Note that the above equation satisfies \cite[Chapter 4, Axiom A]{jk2} with $G = \mathcal{M} \backslash \{ 0 \}$, $A = \kappa$, $\delta = 1$, and $\eta = 1 - 2/({n_{\scaleto{K}{5pt}}+1})$. Thus, by \cite[Chapter 4, Proposition 5.5]{jk2}, we estimate the number of $h$-free ideals with norm $N(\cdot)$ less than or equal to $x$ as the following:
\begin{equation}\label{hfreeidealcount}
    |\mathcal{S}_h(x)| = \frac{\kappa}{\zeta_K(h)} x + O_h \big( R_{\mathcal{S}_h}(x) \big),
\end{equation}
    where
    \begin{equation}\label{RSh(x)}
    R_{\mathcal{S}_h}(x) = \begin{cases}
    x^{1 - \frac{2}{n_{\scaleto{K}{3pt}}+1}}  & \text{ if } \frac{1}{h} < 1 - \frac{2}{n_{\scaleto{K}{3pt}} + 1}, \\
    x^{1 - \frac{2}{n_{\scaleto{K}{3pt}}+1}} (\log x) & \text{ if } \frac{1}{h} = 1 - \frac{2}{n_{\scaleto{K}{3pt}} + 1}, \\
    x^{\frac{1}{h}} & \text{ if } \frac{1}{h} > 1 - \frac{2}{n_{\scaleto{K}{3pt}} + 1}.
\end{cases}
    \end{equation}
% \textcolor{red}{\begin{thm}\label{hfreeideals}
%     Let $x > 1$ be any real number and $h \geq 2$ be any integer. Let $\mathcal{S}_h(x)$ denote the set of $h$-free ideals with norm $N(\cdot)$ less than or equal to $x$. We have
%     $$|\mathcal{S}_h(x)| = \frac{\kappa}{\zeta_K(h)} x + O_h \big( R_{\mathcal{S}_h}(x) \big),$$
%     where
%     \begin{equation*}
%     R_{\mathcal{S}_h}(x) = \begin{cases}
%     x^{1 - \frac{2}{n_{\scaleto{K}{3pt}}+1}}  & \text{ if } \frac{1}{h} < 1 - \frac{2}{n_{\scaleto{K}{3pt}} + 1}, \\
%     x^{1 - \frac{2}{n_{\scaleto{K}{3pt}}+1}} (\log x) & \text{ if } \frac{1}{h} = 1 - \frac{2}{n_{\scaleto{K}{3pt}} + 1}, \\
%     x^{\frac{1}{h}} & \text{ if } \frac{1}{h} > 1 - \frac{2}{n_{\scaleto{K}{3pt}} + 1}.
% \end{cases}
%     \end{equation*}
% \end{thm}}
\begin{rmk}\label{remark1}
    In this work, for convenience, we shall use $R_{\mathcal{S}_h}(x) \ll x^{\tau}$ for some $\tau < 1$  which is evident from the above result. Also, note that, if $n_K = 1$, then $K = \mathbb{Q}$, $\zeta_\mathbb{Q}$ is the classical Riemann $\zeta$-function, and $R_{\mathcal{S}_h}(x) = x^{1/h}$ which matches the distribution of $h$-free natural numbers (see \cite[(4)]{jala}).
\end{rmk}
For the distribution of $h$-full ideals, we prove:
\begin{thm}\label{hfullideals}
Let $x > 1$ be any real number and $h \geq 2$ be any integer. Let $\mathcal{N}_h(x)$ denote the set of $h$-full ideals with norm $N(\cdot)$ less than or equal to $x$. We have
$$|\mathcal{N}_h(x)| = \kappa \gamma_{\scaleto{h}{4.5pt}} x^{1/h} + O_h \big( R_{\mathcal{N}_h}(x) \big),$$
where $\gamma_{\scaleto{h}{4.5pt}}$ is the constant defined as
\begin{equation}\label{gammahk}
    \gamma_{\scaleto{h}{4.5pt}} = \gamma_{\scaleto{h,K}{5.5pt}} := \prod_\wp \left( 1 + \frac{N(\wp) - N(\wp)^{1/h}}{N(\wp)^2 \left( N(\wp)^{1/h} - 1 \right)}\right),
\end{equation}
and where
\begin{equation}\label{E2(x)}
    R_{\mathcal{N}_h}(x) = \begin{cases}
        x^{\frac{1}{h} \left( 1 - \frac{2}{n_{\scaleto{K}{2.5pt}}+1}\right)} & \text{ if } \frac{h}{h+1} < 1 - \frac{2}{n_{\scaleto{K}{3pt}} + 1}, \\
        x^{\frac{1}{h+1}} (\log x) & \text{ if } \frac{h}{h+1} = 1 - \frac{2}{n_{\scaleto{K}{3pt}} + 1}, \\
        x^{\frac{1}{h+1}} & \text{ if } \frac{h}{h+1} > 1 - \frac{2}{n_{\scaleto{K}{3pt}} + 1}.
    \end{cases} 
\end{equation}
\end{thm}
\begin{rmk}\label{remark2}
    In this work, again for convenience, we shall use $R_{\mathcal{N}_h}(x) \ll x^{\xi/h}$ for some $\xi < 1$  which is evident from \thmref{hfullideals}. Also, note that, for $n_K =1$, $R_{\mathcal{N}_h}(x) = x^{1/(h+1)}$ which matches the distribution of $h$-full numbers from the work of Ivi\'c and Shiu (see \cite[Lemma 1]{is}).
\end{rmk}
Finally, using the above distribution results, we prove that $\omega(\mathfrak{m})$ obeys the Gaussian distribution over $h$-free and over $h$-full ideals as well. We achieve this as an application of \thmref{yurugen}. In particular, we establish the following two theorems:
\begin{thm}\label{erdoskachfreeideal}
Let $x > 1$ be any real number and $h \geq 2$ be any integer. Let $\mathcal{S}_h(x)$ denote the set of $h$-free ideals with norm $N(\cdot)$ less than or equal to $x$. Then for any $\gamma \in \mathbb{R}$, we have
$$\lim_{x \rightarrow \infty} \frac{1}{|\mathcal{S}_h(x)|} \left| \left\{ \mathfrak{m} \in \mathcal{S}_h(x) \ : \ N(\mathfrak{m}) \geq 3, \ \frac{\omega(\mathfrak{m}) - \log \log N(\mathfrak{m})}{\sqrt{\log \log N(\mathfrak{m})}} \leq \gamma \right\} \right| = \Phi(\gamma).$$
\end{thm}
\begin{thm}\label{erdoskachfullideal}
Let $x > 1$ be any real number and $h \geq 2$ be any integer. Let $\mathcal{N}_h(x)$ denote the set of $h$-full ideals with norm $N(\cdot)$ less than or equal to $x$. Then for any $\gamma \in \mathbb{R}$, we have
$$\lim_{x \rightarrow \infty} \frac{1}{|\mathcal{N}_h(x)|} \left| \left\{ \mathfrak{m} \in \mathcal{N}_h(x) \ : \ N(\mathfrak{m}) \geq 3, \ \frac{\omega(\mathfrak{m}) - \log \log N(\mathfrak{m})}{\sqrt{\log \log N(\mathfrak{m})}} \leq \gamma \right\} \right| = \Phi(\gamma).$$
\end{thm}

\section{Review of Probability Theory}
In this section, we review some results from probability theory that are essential for our study. Interested readers can find a more detailed version of the results mentioned in this section in the third author's work \cite[Section 2]{liu}.

Let $X$ be a random variable with a probability measure $P$. For a real number $t$, let $F(t)$ be the distribution function of $X$ defined as 
$$F(t) = P(X \leq t).$$
The expectation of $X$ is defined as
$$\textnormal{E}(X) := \int_{-\infty}^\infty t \ d F(t).$$
The variance of $X$, denoted as Var($X$), which measures the deviation of $X$ from its expectation is defined as
$$\textnormal{Var}(X) := \textnormal{E}(X^2) - (\textnormal{E}(X))^2.$$
Moreover, if $Y$ is another random variable with the same probability measure $P$, we have
$$\textnormal{E}(X+Y) = \textnormal{E}(X) + \textnormal{E}(Y).$$
The above property is called the linearity of expectation. Additionally, if $X$ and $Y$ are independent, i.e., for all $x \in X$ and for all $y \in Y$, 
$$P(X \leq x, Y \leq y) = P(X \leq x) \cdot P(Y \leq y),$$
then we have
$$\textnormal{E}(X \cdot Y) = \textnormal{E}(X) \cdot \textnormal{E}(Y),$$
and
$$\textnormal{Var}(X+Y) = \textnormal{Var}(X) + \textnormal{Var}(Y).$$
Given a sequence of random variables $\{X_n\}$ and $\alpha \in \mathbb{R}$, we say $\{X_n\}$ \textit{converges in probability} to $\alpha$ if for any $\epsilon > 0$,
$$\lim_{n \rightarrow \infty} P(|X_n - \alpha| > \epsilon) = 0.$$
We denote this by
$$X_n \xlongrightarrow{P}\alpha.$$
Using the above definitions, we list the following facts from probability theory as mentioned in the third author's work \cite[Page 595-596]{liu}.
\begin{fact}\label{fact1}
    Given a sequence of random variables $\{X_n\}$, if 
    $$\lim_{n \rightarrow \infty} \textnormal{E} (|X_n|) = 0,$$
    we have
    $$X_n \xlongrightarrow{P} 0.$$
\end{fact}
\begin{fact}\label{fact2}
    Let $\{ X_n \}$, $\{ Y_n \}$, and $\{ U_n \}$ be sequences of random variables with the same probability measure P. Let $U$ be a distribution function. Suppose
    $$X_n \xlongrightarrow{P} 1 \quad \textnormal{ and } \quad Y_n \xlongrightarrow{P} 0.$$
    For all $\gamma \in \mathbb{R}$, we have
    $$\lim_{n \rightarrow \infty} P(U_n \leq \gamma) = U(\gamma)$$
    if and only if
    $$\lim_{n \rightarrow \infty} P\left( (X_n U_n + Y_n) \leq \gamma \right) = U(\gamma).$$
\end{fact}
% Let $\Phi(\gamma)$ denote the Gaussian normal distribution defined as
% \begin{equation}\label{Phi}
%     \Phi(\gamma) := \frac{1}{\sqrt{2 \pi}} \int_{-\infty}^\gamma e^{-t^2/2} \ dt.
% \end{equation}
Let $\Phi(\gamma)$ denote the Gaussian normal distribution as defined in \eqref{phi(a)}. For $r \in \mathbb{N}$, the $r$-th moment of $\Phi$ is defined as
$$\mu_r := \int_{-\infty}^\infty t^r d \Phi(t).$$
Then we have:
\begin{fact}\label{fact3}
    Given a sequence of distribution functions $\{F_n\}$, if for all $r \in \mathbb{N}$,
    $$\lim_{n \rightarrow \infty} \int_{-\infty}^\infty t^r d F_n(t) = \mu_r,$$
    then for all $\gamma \in \mathbb{R}$, we have
    $$\lim_{n \rightarrow \infty} F_n(\gamma) = \Phi(\gamma).$$
\end{fact}
As a converse of the above fact, we have
\begin{fact}\label{fact4}
    Let $r \in \mathbb{N}$. Given a sequence of distribution functions $\{F_n\}$, if 
    $$\lim_{n \rightarrow \infty} F_n(\gamma) = \Phi(\gamma), \quad \textnormal{for all } \gamma \in \mathbb{R}$$
    and
    $$\sup_n \left\{ \int_{-\infty}^\infty |t|^{r+\delta} dF_n(t) \right\} < \infty, \quad \textnormal{for some } \delta = \delta(r) > 0,$$
    we have
    $$\lim_{n \rightarrow \infty} \int_{-\infty}^\infty t^r d F_n(t) = \mu_r.$$
\end{fact}
The next fact is a special case of the Central Limit Theorem.
\begin{fact}\label{fact5}
    Let $X_1,X_2,\ldots, X_i, \ldots$ be a sequence of independent random variables and $\textnormal{Im}(X_i)$ is the image of $X_i$. Suppose
    \begin{enumerate}
        \item $\sup_{i} \{ \textnormal{Im}(X_i) \} < \infty$,
        \item $\textnormal{E}(X_i) = 0$ and $\textnormal{Var}(X_i) < \infty$ for all $i$.
    \end{enumerate}
    For $n \in \mathbb{N}$, let $\Phi_n$ be the normalization of $X_1,X_2,\ldots,X_n$ defined as
    $$\Phi_n := \left( \sum_{i=1}^{n} X_i \right) \bigg\slash \left( \sum_{i=1}^n \textnormal{Var}(X_i) \right)^{1/2}.$$
    If $\sum_{i=1}^\infty \textnormal{Var}(X_i)$ diverges, then we have
    $$\lim_{n \rightarrow \infty} P(\Phi_n \leq \gamma) = \Phi(\gamma).$$
\end{fact}
\section{Lemmata}
In this section, we list all the lemmas required to prove our theorems. The first three lemmas establish statements equivalent to \thmref{yurugen}, and thus proving any equivalent statement would be sufficient in proving the theorem. The next two lemmas establish results necessary to prove one of the equivalent conditions mentioned in the third lemma of this section. Together, these lemmas prove \thmref{yurugen} in the next section. The final lemma in this section establishes results involving ideals required to complete the proofs of \thmref{hfullideals}, \thmref{erdoskachfreeideal}, and \thmref{erdoskachfullideal}.

Let $\mathcal{P}$, $\mathcal{M}$, and $\mathcal{S}$ be defined as in Section 1 and assume that they satisfy \eqref{Scondition} and the conditions (a) to (f). For $\mathfrak{m} \in \mathcal{S}$ and $x > 1$, we define
$$P_{\mathcal{S},x} \{ \mathfrak{m} \ : \ \mathfrak{m} \text{ satisfies some conditions} \}$$
to be the quantity 
$$\frac{1}{|\mathcal{S}(x)|} \left| \{ \mathfrak{m} \in \mathcal{S}(x) \ : \  \mathfrak{m} \text{ satisfies some conditions} \} \right|.$$
Note that $P_{\mathcal{S},x}$ is a probability measure on $\mathcal{S}$. Let $g$ be a function from $\mathcal{S}$ to $\mathbb{R}$. The expectation of $g$ with respect to $P_{\mathcal{S},x}$ is denoted by
$$\textnormal{E}_{\mathcal{S},x} \{ g \} := \frac{1}{|\mathcal{S}(x)|} \sum_{\mathfrak{m} \in \mathcal{S}(x)} g(\mathfrak{m}).$$

The first lemma gives an equivalent statement of \thmref{yurugen}.
\begin{lma}\label{linkN(m)x}
    $$\lim_{x \rightarrow \infty} P_{\mathcal{S},x} \bigg\{ \mathfrak{m} \ : \ N(\mathfrak{m}) \geq 3, \ \frac{\omega(\mathfrak{m}) - \log \log N(\mathfrak{m})}{\sqrt{\log \log N(\mathfrak{m})}} \leq \gamma \bigg\} = \Phi(\gamma)$$
    if and only if
    $$\lim_{x \rightarrow \infty} P_{\mathcal{S},x} \bigg\{ \mathfrak{m} \ : \ \frac{\omega(\mathfrak{m}) - \log \log x}{\sqrt{\log \log x}} \leq \gamma \bigg\} = \Phi(\gamma).$$
\end{lma}
\begin{proof}
The proof closely follows the steps of the proof of \cite[Lemma 3]{liu}. First note that
\begin{align*}
    \frac{\omega(\mathfrak{m}) - \log \log x}{\sqrt{\log \log x}} & = \frac{\omega(\mathfrak{m}) - \log \log N(\mathfrak{m})}{\sqrt{\log \log N(\mathfrak{m})}} \frac{\sqrt{\log \log N(\mathfrak{m})}}{\sqrt{\log \log x}} \\
    & \hspace{.5cm} + \frac{\log \log N(\mathfrak{m}) - \log \log x}{\sqrt{\log \log x}}.
\end{align*}
Thus by Fact 2 and our assumption that $\mathcal{S}$ is infinite, to prove the lemma, it suffices to show that for any $\epsilon >0$, 
$$\lim_{x \rightarrow \infty} P_{\mathcal{S},x} \bigg\{ \mathfrak{m} \ : \ N(\mathfrak{m}) \geq 3, \ \bigg| \frac{\sqrt{\log \log N(\mathfrak{m})}}{\sqrt{\log \log x}} - 1 \bigg| > \epsilon \bigg\} = 0 $$
and
$$\lim_{x \rightarrow \infty} P_{\mathcal{S},x} \bigg\{ \mathfrak{m} \ : \ N(\mathfrak{m}) \geq 3, \ \bigg| \frac{\log \log N(\mathfrak{m}) - \log \log x}{\sqrt{\log \log x}} \bigg| > \epsilon \bigg\} = 0.$$
Consider $\mathfrak{m} \in \mathcal{M}$ satisfying $3 \leq x^{1/2} < N(\mathfrak{m}) \leq x$. If we have
$$\frac{\sqrt{\log \log N(\mathfrak{m})}}{\sqrt{\log \log x}} < 1 - \epsilon,$$
it follows that
$$(\log \log x - \log 2)^{1/2} < (\log \log N(\mathfrak{m}))^{1/2} < (1 - \epsilon) (\log \log x)^{1/2}.$$
Taking squares on both the sides above, we get
$$\frac{1}{(1-\epsilon)^2} (\log \log x - \log 2) < \log \log x  \quad \longleftrightarrow \quad \log \log x < \frac{\log 2}{\epsilon (2 - \epsilon)}.$$
Similarly, for $\mathfrak{m} \in \mathcal{M}$ satisfying $3 \leq x^{1/2} < N(\mathfrak{m}) \leq x$, if we have
$$\frac{\log \log x - \log \log N(\mathfrak{m})}{\sqrt{\log \log x}} > \epsilon,$$
it implies that
$$\log \log x < \left( \frac{\log 2}{\epsilon} \right)^2.$$
Thus, there exists sufficiently large $x(\epsilon) \in X$ such that for all $x \in X$ with $x \geq x(\epsilon)$, we have
$$P_{\mathcal{S},x} \bigg\{ \mathfrak{m} \ : \ N(\mathfrak{m}) \geq 3, \ \left| \frac{\sqrt{\log \log N(\mathfrak{m})}}{\sqrt{\log \log x}} - 1 \right| > \epsilon \bigg\} \leq P_{\mathcal{S},x} \{ \mathfrak{m} \ : \ N(\mathfrak{m}) \leq x^{1/2} \}$$
and
$$P_{\mathcal{S},x} \bigg\{ \mathfrak{m} \ : \ N(\mathfrak{m}) \geq 3, \ \bigg| \frac{\log \log N(\mathfrak{m}) - \log \log x}{\sqrt{\log \log x}} \bigg| > \epsilon \bigg\} \leq P_{\mathcal{S},x} \{ \mathfrak{m} \ : \ N(\mathfrak{m}) \leq x^{1/2} \}.$$
Applying \eqref{Scondition}, we obtain
$$P_{\mathcal{S},x} \{ \mathfrak{m} \ : \ N(\mathfrak{m}) \leq x^{1/2} \} = o(1).$$
This completes the proof.
\end{proof}
\begin{rmk}
    Note that, we need to use condition \eqref{Scondition} to complete the last step of the previous proof. In fact, this is the only place where \eqref{Scondition} is applied. Moreover, the lemma still holds if 1/2 is replaced by any constant between 0 and 1. Thus, we can replace \eqref{Scondition} by
    $$|\mathcal{S}(x^c)| = o(|\mathcal{S}(x)|),$$
    where $0 < c < 1$.
\end{rmk}

Let $\beta$ be a constant with $0 < \beta \leq 1$ and $y = y(x) < x^\beta$ satisfying the conditions (a)-(f) as mentioned in Section 1. For $\mathfrak{m} \in \mathcal{M}$, we define the truncated function
$$\omega_y(\mathfrak{m}) = \left| \{ \wp \in \mathcal{P} \ : \ N(\wp) \leq y, \ \wp | \mathfrak{m} \} \right|.$$

The next result establishes another equivalent formulation of the Erd\H{o}s-Kac theorem in terms of $\omega_y$. 

\begin{lma}\label{linkomegay}
    $$\lim_{x \rightarrow \infty} P_{\mathcal{S},x} \bigg\{ \mathfrak{m} \ : \ \frac{\omega(\mathfrak{m}) - \log \log x}{\sqrt{\log \log x}} \leq \gamma \bigg\} = \Phi(\gamma)$$
    if and only if
    $$\lim_{x \rightarrow \infty} P_{\mathcal{S},x} \bigg\{ \mathfrak{m} \ : \ \frac{\omega_y(\mathfrak{m}) - \log \log x}{\sqrt{\log \log x}} \leq \gamma \bigg\} = \Phi(\gamma).$$
\end{lma}
\begin{proof}
    Note that
    $$\frac{\omega_y(\mathfrak{m}) - \log \log x}{\sqrt{\log \log x}} = \frac{\omega(\mathfrak{m}) - \log \log x}{\sqrt{\log \log x}} + \frac{\omega_y(\mathfrak{m}) - \omega(\mathfrak{m})}{\sqrt{\log \log x}}.$$
    Thus, by Fact 1 and Fact 2, to prove the lemma, it suffices to prove
    $$\lim_{x \rightarrow \infty} \textnormal{E}_{\mathcal{S},x} \bigg\{ \mathfrak{m} \ : \ \bigg| \frac{\omega(\mathfrak{m}) - \omega_y(\mathfrak{m})}{\sqrt{\log \log x}} \bigg| \bigg\} = 0.$$ 
    Notice that
    \begin{align*}
        & \sum_{\substack{\mathfrak{m} \in \mathcal{S} \\ N(\mathfrak{m}) \leq x}}  |\omega(\mathfrak{m}) - \omega_y(\mathfrak{m})| \\
        & = \sum_{\substack{\mathfrak{m} \in \mathcal{S} \\ N(\mathfrak{m}) \leq x}}  \sum_{\substack{\wp \in \mathcal{P} \\ N(\wp) > y, \ \wp | \mathfrak{m} }} 1  \\
        & = \sum_{\substack{\wp \in \mathcal{P} \\ y < N(\wp) \leq x^\beta }} \sum_{\substack{\mathfrak{m} \in \mathcal{S} \\ N(\mathfrak{m}) \leq x, \ \wp | \mathfrak{m}}} 1 + \sum_{\substack{\mathfrak{m} \in \mathcal{S} \\ N(\mathfrak{m}) \leq x}}  \sum_{\substack{\wp \in \mathcal{P} \\ N(\wp) > x^\beta, \ \wp | \mathfrak{m}}} 1.
    \end{align*}
Using the definition of $\lambda_\wp$ and $e_\wp$, and the conditions (a), (b), and (c), we obtain
\begin{align*}
    \sum_{\substack{\mathfrak{m} \in \mathcal{S} \\ N(\mathfrak{m}) \leq x}}  |\omega(\mathfrak{m}) - \omega_y(\mathfrak{m})| & = \sum_{\substack{\wp \in \mathcal{P} \\ y < N(\wp) \leq x^\beta }} |\mathcal{S}(x)|(\lambda_\wp+ e_\wp) + O(|\mathcal{S}(x)|) \\
    & = o(|\mathcal{S}(x)|(\log \log x)^{1/2}) + O(|\mathcal{S}(x)|).
\end{align*}
Thus, we have
$$\textnormal{E}_{\mathcal{S},x} \bigg\{ \mathfrak{m} \ : \ \bigg| \frac{\omega(\mathfrak{m}) - \omega_y(\mathfrak{m})}{\sqrt{\log \log x}} \bigg| \bigg\} = \frac{o(|\mathcal{S}(x)|(\log \log x)^{1/2})}{|\mathcal{S}(x)|(\log \log x)^{1/2}} = o(1),$$
which completes the proof.
\end{proof}
For $\wp \in \mathcal{P}$, let $\lambda_{\wp}$ be the constant (depending on $\wp$) mentioned in \eqref{eq_lambdap}. We define the independent random variable $X_\wp$ by
$$P(X_\wp =1) = \lambda_\wp$$
and
$$P(X_\wp=0) = 1 - \lambda_\wp.$$
We define a new random variable $\mathcal{S}_y$ by
$$\mathcal{S}_y := \sum_{N(\wp) \leq y} X_\wp.$$
Note that, by conditions (d) and (e), we have the expectation and variance of the random variable $\mathcal{S}_y$ as
$$\text{E}(\mathcal{S}_y) = \sum_{N(\wp) \leq y} \lambda_\wp = \log \log x + o \left( (\log \log x)^{1/2} \right),$$
and
$$\text{Var}(\mathcal{S}_y) = \sum_{N(\wp) \leq y} \lambda_\wp(1 - \lambda_\wp) = \log \log x + o \left( (\log \log x)^{1/2} \right).$$
Note that, we will use the notation $\textnormal{E}(\cdot)$ and $\textnormal{E}_{S,x}\{\cdot\}$ respectively to distinguish the expectation of a random variable from the expectation of a function with respect to $P_{S,x}$. However, in most cases, they will represent the same values.

The above setup leads us to another reformulation of \thmref{yurugen} in terms of $\textnormal{E}(\mathcal{S}_y)$.
\begin{lma}\label{linkomegayESy}
$$\lim_{x \rightarrow \infty} P_{\mathcal{S},x} \bigg\{ m \ : \ \frac{\omega_y(\mathfrak{m}) - \log \log x}{\sqrt{\log \log x}} \leq \gamma \bigg\} = \Phi(\gamma)$$
if and only if
$$\lim_{x \rightarrow \infty} P_{\mathcal{S},x} \bigg\{ \mathfrak{m} \ : \ \frac{\omega_y(\mathfrak{m}) - \textnormal{E}(\mathcal{S}_y)}{\sqrt{\textnormal{Var}(\mathcal{S}_y)}} \leq \gamma \bigg\} = \Phi(\gamma).$$
\end{lma}
\begin{proof}
Note that
$$\frac{\omega_y(\mathfrak{m}) - \textnormal{E}(\mathcal{S}_y)}{\sqrt{\textnormal{Var}(\mathcal{S}_y)}} = \frac{\omega_y(\mathfrak{m}) - \log \log x}{\sqrt{\log \log x}} \frac{\sqrt{\log \log x}}{\sqrt{\textnormal{Var}(\mathcal{S}_y)}} + \frac{\log \log x - \textnormal{E}(\mathcal{S}_y)}{\sqrt{\textnormal{Var}(\mathcal{S}_y)}}.$$
Since
$$\text{Var}(\mathcal{S}_y) = \sum_{N(\wp) \leq y} \lambda_\wp(1 - \lambda_\wp) = \log \log x + o \left( (\log \log x)^{1/2} \right),$$
we have
$$\frac{\sqrt{\log \log x}}{\sqrt{\textnormal{Var}(\mathcal{S}_y)}} \xlongrightarrow{P} 1,$$
where $\xlongrightarrow{P}$ denotes the convergence in probability. Moreover, since 
$$\text{E}(\mathcal{S}_y) = \sum_{N(\wp) \leq y} \lambda_\wp = \log \log x + o \left( (\log \log x)^{1/2} \right),$$
we obtain
$$\lim_{x \rightarrow \infty} \textnormal{E}_{\mathcal{S},x} \bigg\{ \mathfrak{m} \ : \  \bigg| \frac{\textnormal{E}(\mathcal{S}_y) - \log \log x}{\sqrt{\textnormal{Var}(\mathcal{S}_y)}} \bigg| \bigg\} = 0.$$
Finally, by using Fact 1 and Fact 2, we complete the proof of the equivalence mentioned in the lemma.
\end{proof}
Next, we introduce another set of random variables. For $\wp \in \mathcal{P}$, we define a random variable $\delta_\wp: \mathcal{M} \rightarrow \mathbb{R}$ by
$$\delta_\wp(\mathfrak{m}) := \begin{cases}
    1 & \text{if } \wp | \mathfrak{m}, \\
    0 & \text{otherwise.}
\end{cases}$$
Thus, we can write
$$\omega_y(\mathfrak{m}) = \sum_{\substack{\wp \in \mathcal{P} \\ N(\wp) \leq y, \ \wp | \mathfrak{m} }} 1 = \sum_{\substack{\wp \in \mathcal{P} \\ N(\wp) \leq y}} \delta_\wp(\mathfrak{m}).$$
Notice that for a fixed $\wp \in \mathcal{P}$ and $x \in X$, by definition, we have
$$P_{\mathcal{S},x} \{ \mathfrak{m} \ : \ \delta_\wp(\mathfrak{m}) = 1 \} = \lambda_\wp + e_\wp.$$
Since the expectations of random variables $X_\wp$ and $\delta_\wp$ are close, the sum $\mathcal{S}_y$ is a good approximation of $\omega_y$. Indeed, the $r$-th moments of their normalizations are equal as $x \rightarrow \infty$, which we prove in the following result.
\begin{lma}\label{rthmoment}
 Let $r \in \mathbb{N}$. We have
 $$\lim_{x \rightarrow \infty} \bigg| \textnormal{E}_{\mathcal{S},x} \bigg\{ \bigg( \frac{\omega_y(\mathfrak{m}) - \textnormal{E}(\mathcal{S}_y)}{\sqrt{\textnormal{Var}(\mathcal{S}_y)}} \bigg)^r \bigg\}- E\bigg( \bigg( \frac{\mathcal{S}_y - \textnormal{E}(\mathcal{S}_y)}{\sqrt{\textnormal{Var}(\mathcal{S}_y)}} \bigg)^r\bigg) \bigg| = 0.$$
\end{lma}
\begin{proof}
    For $0 \leq k \leq r$, we write
    $$\textnormal{E}(\mathcal{S}_y^k) = \sum_{u=1}^k \left( \sum{\vphantom{\sum}}' \frac{k!}{k_1! \cdots k_u!} \left( \sum{\vphantom{\sum}}'' \textnormal{E}(X_{\wp_{\scaleto{1}{3pt}}}^{k_1} \cdots X_{\wp_{\scaleto{u}{3pt}}}^{k_u}) \right) \right),$$
    where $\sum{\vphantom{\sum}}'$ extends over all $u$-tuples $(k_1,k_2,\cdots,k_u)$ of positive integers such that $k_1+k_2+\cdots+k_u = k$ and $\sum{\vphantom{\sum}}''$ extends over all $u$-tuples $(\wp_{\scaleto{1}{3pt}},\wp_{\scaleto{2}{3pt}},\cdots,\wp_{\scaleto{u}{3pt}})$ with $N(\wp_i) \leq y$ and $\wp_i'$s are distinct. Since each $X_i$ takes values 0 or 1 and the $X_{\wp_i}'$s are independent, we have 
    $$\textnormal{E}(X_{\wp_{\scaleto{1}{3pt}}}^{k_1} \cdots X_{\wp_{\scaleto{u}{3pt}}}^{k_u}) = \lambda_{\wp_{\scaleto{1}{3pt}}}\cdot \lambda_{\wp_{\scaleto{2}{3pt}}} \cdot \cdots \cdot \lambda_{\wp_{\scaleto{u}{3pt}}}.$$
    Similarly, we have
    $$\textnormal{E}_{\mathcal{S},x} \{ \omega_y^k \} = \sum_{u=1}^k \left( \sum{\vphantom{\sum}}' \frac{k!}{k_1! \cdots k_u!} \left( \sum{\vphantom{\sum}}'' \textnormal{E}_{\mathcal{S},x} \{ \delta_{\wp_{\scaleto{1}{3pt}}}^{k_1} \cdots \delta_{\wp_{\scaleto{u}{3pt}}}^{k_u} \} \right) \right),$$
    with $\sum{\vphantom{\sum}}'$ and $\sum{\vphantom{\sum}}''$ defined as above. Notice that, by the definition of $\lambda_\wp$ and $e_\wp$, we have
    $$|\textnormal{E}(X_{\wp_{\scaleto{1}{3pt}}}^{k_1} \cdots X_{\wp_{\scaleto{u}{3pt}}}^{k_u}) - \textnormal{E}_{\mathcal{S},x} \{ \delta_{\wp_{\scaleto{1}{3pt}}}^{k_1} \cdots \delta_{\wp_{\scaleto{u}{3pt}}}^{k_u} \}| = |e_{\wp_{\scaleto{1}{3pt}} \cdots \wp_{\scaleto{u}{3pt}}}|.$$
    Next, we write
    $$\textnormal{E}((\mathcal{S}_y - E(\mathcal{S}_y))^r) = \sum_{k=0}^r \binom{r}{k} \textnormal{E}(\mathcal{S}_y^k) \cdot \textnormal{E}(\mathcal{S}_y)^{r-k}$$
    and
    $$\textnormal{E}_{\mathcal{S},x}\{(\omega_y - E(\mathcal{S}_y))^r \} = \sum_{k=0}^r \binom{r}{k} \textnormal{E}_{\mathcal{S},x} \{ \omega_y^k \} \cdot \textnormal{E}(\mathcal{S}_y)^{r-k}.$$
    Since
    $$\textnormal{E}(\mathcal{S}_y) = \log \log x + o((\log \log x)^{1/2}),$$
    using Condition (f), we obtain
    \begin{align*}
        & | \textnormal{E}_{\mathcal{S},x}\{(\omega_y - E(\mathcal{S}_y))^r \} - \textnormal{E}((\mathcal{S}_y - E(\mathcal{S}_y))^r)| \\
        & \ll \sum_{k=0}^r \binom{r}{k} \left( \sum_{u=1}^k \left( \sum{\vphantom{\sum}}' \frac{k!}{k_1! \cdots k_u!} \left( \sum{\vphantom{\sum}}'' |e_{\wp_{\scaleto{1}{3pt}} \cdots \wp_{\scaleto{u}{3pt}}}| \cdot (\log \log x)^{r-k} \right) \right) \right)\\
        & = o((\log \log x)^{r/2}).
    \end{align*}
    Since
    $$\text{Var}(\mathcal{S}_y) = \log \log x + o \left( (\log \log x)^{1/2} \right),$$
    the lemma follows.    
\end{proof}
The next result is about the $r$-th moment of the random variable $\mathcal{S}_y$.
\begin{lma}\label{finite_r-thmoment}
    For $r \in \mathbb{N}$,
    $$\sup_y \bigg| \textnormal{E} \bigg( \bigg(\frac{\mathcal{S}_y - \textnormal{E}(\mathcal{S}_y)}{\sqrt{\textnormal{Var}(\mathcal{S}_y)}} \bigg)^r \bigg) \bigg| < \infty.$$
\end{lma}
\begin{proof}
    We define $Y_\wp = X_\wp - \lambda_\wp$. We have
    $$\textnormal{E}((\mathcal{S}_y - E(\mathcal{S}_y))^r) = \sum_{u=1}^r \left( \sum{\vphantom{\sum}}' \frac{k!}{k_1! \cdots k_u!} \left( \sum{\vphantom{\sum}}'' \textnormal{E} (Y_{\wp_{\scaleto{1}{3pt}}}^{k_1} \cdots Y_{\wp_{\scaleto{u}{3pt}}}^{k_u})\right) \right),$$
with $\sum{\vphantom{\sum}}'$ and $\sum{\vphantom{\sum}}''$ defined as in last \lmaref{rthmoment}, except replacing $k$ by $r$. Then, following the exact steps in the proof of \cite[Lemma 7]{liu}, but using $\wp$ instead of $p$ and $\wp_i'$s instead of $p_i'$s, we complete the proof.
\end{proof}

Next, we recall the following results regarding sums over ideals necessary for our study. 
% \textcolor{red}{For $x > 1$, let $I_K(x)$ count the number of ideals $\mathfrak{m} \in \mathcal{M}$ with $N(\mathfrak{m} ) \leq x$. }
% % Weber in \cite[Pages 84-85]{weber} proved that
% % $$I_K(x) = \kappa x + O \left( x^{1 - \frac{1}{n_{\scaleto{K}{2.5pt}}}}\right),$$
% % where 
% % $$\kappa = \kappa_K = \frac{2^{r_1}(2 \pi)^{r_2} h R}{\nu \sqrt{|d_K|}},$$
% % with
% % \begin{align*}
% %     r_1 & = \text{the number of real embeddings of } K,\\
% %     2 r_2 & = \text{the number of complex embeddings of } K,\\
% %     h & = \text{the class number},\\
% %     R & = \text{the regulator}, \\
% %     \nu & = \text{the number of roots of unity}, \\
% %     d_K & = \text{the discriminant of } K.
% % \end{align*}
% \textcolor{red}{Landau in \cite[Satz 210]{Landau} proved that
% \begin{equation*}
%     I_K(x) = \kappa x + O \left( x^{1 - \frac{2}{n_{\scaleto{K}{3pt}}+1}}\right),
% \end{equation*}
% where $\kappa$ is defined in \eqref{defkappa}.} 
For $x > 1$, let $\pi_K(x)$ count the number of prime ideals $\wp \in \mathcal{P}$ with norm $N(\cdot)$ less than or equal to $x$. Landau in \cite{Landaupnt} generalized the classical prime number theorem to number fields by proving
\begin{equation*}
    \pi_K(x) = \text{Li}(x) + O \left( x \exp (-c_K \sqrt{\log x})\right),
\end{equation*}
for some computable constant $c_K$, where $\text{Li}(x) = \int_{2}^x dt/\log t$ is the offset logarithmic integral.
From this, we can deduce
\begin{equation}\label{pnt_num}
    \pi_K(x) = \frac{x}{\log x} + O \left( \frac{x}{(\log x)^2} \right).
\end{equation}
Recall that for any $s \in \mathbb{C}$ with $\Re(s) > 1$, the Dedekind zeta-function, $\zeta_K(s)$ is defined as
\begin{equation*}
    \zeta_K(s) = \sum_{\substack{\mathfrak{m}  \in \mathcal{M} \\ \mathfrak{m} \neq 0}} \frac{1}{N(\mathfrak{m} )^s}= \prod_{\wp \in \mathcal{P}} \left( 1 - \frac{1}{N(\wp)^s} \right)^{-1}.
\end{equation*}
Note that the above Euler product formula relates the sum representation of $\zeta_K$ to its product representation. Also, note that $\zeta_K(s)$ has a simple pole at $s=1$ with residue $\kappa$. Using the above definitions and results, we prove:
\begin{lma}\label{boundnm}
   Let $\alpha$ be a real number. We have
   \begin{enumerate}
       \item If $0 \leq \alpha < 1$, 
       $$\sum_{\substack{\wp \in \mathcal{P} \\ N(\wp) \leq x}} \frac{1}{N(\wp)^\alpha} \ll \frac{x^{1- \alpha}}{\log x}.$$
       \item If $0 \leq \alpha < 1$, 
       $$\sum_{\substack{\mathfrak{m} \in  \mathcal{M} \backslash \{ 0 \} \\ N(\mathfrak{m}) \leq x}} \frac{1}{N(\mathfrak{m})^\alpha} \ll x^{1- \alpha}.$$
       \item If $\alpha > 1$, then
       $$\sum_{\substack{\mathfrak{m} \in  \mathcal{M} \backslash \{ 0 \} \\ N(\mathfrak{m}) \leq x}} \frac{1}{N(\mathfrak{m})^\alpha} \ll 1,$$
       and thus
       $$\sum_{\substack{\wp \in \mathcal{P} \\ N(\wp) \leq x}} \frac{1}{N(\wp)^\alpha} \ll 1.$$
       \item As a generalization of Mertens' theorem, we have
       $$\sum_{\substack{\wp \in \mathcal{P} \\ N(\wp) \leq x}} \frac{1}{N(\wp)} = \log \log x + A + O \left( \frac{1}{\log x} \right),$$
       for some constant $A$.
       \item We have
       $$\sum_{\substack{\mathfrak{m} \in \mathcal{M} \backslash \{ 0 \} \\ N(\mathfrak{m}) \leq x}} \frac{1}{N(\mathfrak{m})} = \kappa \log x + B + O \left( \frac{1}{x^{\frac{2}{n_K+1} }} \right),$$
       for some constant $B$.
       \item If $\alpha > 1$, then
       $$\sum_{\substack{\wp \in \mathcal{P} \\ N(\wp) \geq x}} \frac{1}{N(\wp)^\alpha} = \frac{1}{(\alpha-1)x^{\alpha-1} (\log x)} + O \left( \frac{1}{x^{\alpha-1} (\log x)^2} \right).$$
       \item If $\alpha > 1$, then
       $$\sum_{\substack{\mathfrak{m} \in \mathcal{M} \\ N(\mathfrak{m}) \geq x}} \frac{1}{N(\mathfrak{m})^\alpha} = \frac{\kappa}{(\alpha-1)x^{\alpha-1}} + O \left( \frac{1}{x^{\alpha -1 + \frac{2}{n_{\scaleto{K}{2.5pt}} + 1}}} \right).$$
   \end{enumerate}
\end{lma}
\begin{proof}
    Notice that \eqref{Landaueq} and \eqref{pnt_num} satisfy the conditions (A) and (B) in \cite[Page 574]{liuturan}, and thus Parts 1, 3, and 4 follow from \cite[Lemma 1 and 2]{liuturan}. Also, recall that \eqref{Landaueq} satisfies Axiom A of \cite[Chapter 4]{jk2}. Thus, by \cite[Proposition 2.8(i)]{jk2}, Part 5 follows. Parts 2, 6, and 7 follow from the technique of partial summation (see \cite[Lemma 1.2]{aler}).
\end{proof}

\section{The Erd\H{o}s-Kac theorem over subsets in number fields}
In this section, we prove the Erd\H{o}s-Kac theorem over any subset of the set of ideals $\mathcal{M}$ satisfying the set of conditions mentioned in \thmref{yurugen}.
\begin{proof}[\textbf{Proof of \thmref{yurugen}}]
Given $\mathcal{P}, \mathcal{M}$ and $\mathcal{S}$ as in the statement of the theorem, we assume that for all $x > 1$, there exists a constant $\beta$ with $0 < \beta \leq 1$ and $y = y(x) < x^\beta$ such that the conditions \eqref{Scondition} and (a) to (f) satisfy. For $\mathfrak{m}  \in \mathcal{S}$, we want to show the quantity 
$$\frac{\omega(\mathfrak{m} ) - \log \log N(\mathfrak{m} )}{\sqrt{\log \log N(\mathfrak{m} )}}$$
satisfies the normal distribution. By the equivalent statements in \lmaref{linkN(m)x},  \lmaref{linkomegay}, and \lmaref{linkomegayESy}, to prove \thmref{yurugen}, it suffices to prove
$$\lim_{x \rightarrow \infty} P_{\mathcal{S},x} \bigg\{ \mathfrak{m}  \ : \ \frac{\omega_y(\mathfrak{m}) - \textnormal{E}(\mathcal{S}_y)}{\sqrt{\textnormal{Var}(\mathcal{S}_y)}} \leq \gamma \bigg\} = \Phi(\gamma).$$
The distribution function $F_y$ respect to $P_{\mathcal{S},x}$ is defined by
$$F_y(\gamma) := P_{\mathcal{S},x} \bigg\{ \mathfrak{m}   \ : \ \frac{\omega_y(\mathfrak{m} ) - \textnormal{E}(\mathcal{S}_y)}{\sqrt{\textnormal{Var}(\mathcal{S}_y)}} \leq \gamma \bigg\}.$$
Notice that the $r$-th moment of $F_y$ can be written as
\begin{align*}
    & \int_{-\infty}^\infty t^r dF_y(t) \\
    & = \sum_{t = -\infty}^\infty \bigg\{ \lim_{u \rightarrow \infty} \sum_{i=1}^u (t+i/u)^r \bigg( F_y(t+i/u) - F_y(t+(i-1)/u) \bigg) \bigg\} \\
    & = \sum_{t = -\infty}^\infty \bigg\{ \lim_{u \rightarrow \infty} \sum_{i=1}^u (t+i/u)^r P_{\mathcal{S},x} \bigg\{ \mathfrak{m}  \ : \ (t+(i-1)/u) < \frac{\omega_y(\mathfrak{m}) - \textnormal{E}(\mathcal{S}_y)}{\sqrt{\textnormal{Var}(\mathcal{S}_y)}} \leq (t + i/u) \bigg\} \bigg\}. 
\end{align*}
Thus, by the definition of $P_{\mathcal{S},x}$, we have
$$\int_{-\infty}^\infty t^r dF_y(t) = \frac{1}{|\mathcal{S}(x)|} \sum_{\mathfrak{m}  \in \mathcal{S}(x)} \left( \frac{\omega_y(\mathfrak{m} ) - \textnormal{E}(\mathcal{S}_y)}{\sqrt{\textnormal{Var}(\mathcal{S}_y)}} \right)^r = \textnormal{E}_{\mathcal{S},x} \bigg\{ \bigg( \frac{\omega_y(\mathfrak{m} ) - \textnormal{E}(\mathcal{S}_y)}{\sqrt{\textnormal{Var}(\mathcal{S}_y)}} \bigg)^r \bigg\}.$$
Hence, to prove
$$\lim_{x \rightarrow \infty} F_y(\gamma) = \Phi(\gamma),$$
by Fact 3, it suffices to show that for all $r \in \mathbb{N}$,
$$\lim_{x \rightarrow \infty} \textnormal{E}_{\mathcal{S},x} \bigg\{ \bigg( \frac{\omega_y(\mathfrak{m}) - \textnormal{E}(\mathcal{S}_y)}{\sqrt{\textnormal{Var}(\mathcal{S}_y)}} \bigg)^r \bigg\} = \mu_r.$$
By \lmaref{rthmoment}, we observe that the last inequality holds if
$$\lim_{x \rightarrow \infty} \textnormal{E} \bigg( \bigg( \frac{\mathcal{S}_y(\mathfrak{m} ) - \textnormal{E}(\mathcal{S}_y)}{\sqrt{\textnormal{Var}(\mathcal{S}_y)}} \bigg)^r \bigg) = \mu_r.$$
We define a new random variable $\Phi$ by
$$\Phi_y := \frac{\mathcal{S}_y - \textnormal{E}(\mathcal{S}_y)}{\sqrt{\textnormal{Var}(\mathcal{S}_y)}}.$$
Note that \lmaref{finite_r-thmoment} ensures that any sequence of $\Phi_y$'s satisfies the hypothesis of Fact 5. Thus, by the Central Limit theorem given in Fact 5, we have
$$\lim_{x \rightarrow \infty} P(\Phi_{y(x)} \leq \gamma) = \Phi(\gamma), \quad \text{for all } \gamma \in \mathbb{R}.$$
Also, \lmaref{finite_r-thmoment} implies that for each $r \in \mathbb{N}$, there exists $\delta = \delta(r) > 0$ such that
$$\sup_y \left\{ \int_{-\infty}^\infty |t|^{r+\delta} d\Phi_y(t) \right\} < \infty.$$
Combining the last two observations with Fact 4, we obtain
$$\lim_{x \rightarrow \infty} \textnormal{E} \bigg( \bigg( \frac{\mathcal{S}_y(\mathfrak{m}) - \textnormal{E}(\mathcal{S}_y)}{\sqrt{\textnormal{Var}(\mathcal{S}_y)}} \bigg)^r \bigg) = \mu_r,$$
and thus establish
$$\lim_{x \rightarrow \infty} F_y(\gamma) = \Phi(\gamma).$$
This completes the proof of \thmref{yurugen}, i.e., we obtain that for any $\gamma \in \mathbb{R}$, we have
$$\lim_{x \rightarrow \infty} \frac{1}{|\mathcal{S}(x)|} \bigg| \left\{ \mathfrak{m}  \in \mathcal{S}(x) \ : \ N(\mathfrak{m}) \geq 3, \  \frac{\omega(\mathfrak{m}) - \log \log N(\mathfrak{m})}{\sqrt{\log \log N(\mathfrak{m})}} \leq \gamma\right\} \bigg| = \Phi(\gamma).$$
\end{proof}
\section{The Erd\H{o}s-Kac theorem over \texorpdfstring{$h$}{}-free ideals}
To prove the Erd\H{o}s-Kac theorem over $h$-free ideals given in \thmref{erdoskachfreeideal}, we begin by estimating the distribution of $h$-free ideals which are co-prime to a given prime ideal. To do this, we introduce the generalized M\"obius $\mu$-function. For an $\mathfrak{m} \in \mathcal{M}$, the Mobius function $\mu_K(\mathfrak{m})$ is defined as
$$\mu_K(\mathfrak{m}) = \begin{cases}
    1 & \text{ if } \mathfrak{m} = 1, \\
    (-1)^r & \text{ if } \mathfrak{m} \text{ is a product of $r$ distinct prime ideals},\\
    0 & \text{ if } \mathfrak{m} \text{ is divided by square of a prime ideal}.
\end{cases}
$$
Note that, for any $\mathfrak{m}  \in \mathcal{M}$, $\mu_K(\mathfrak{m})$ satisfies the identity
\begin{equation}\label{iden_mu}
    \sum_{\substack{\mathfrak{d} \in \mathcal{M} \\ \mathfrak{d}^h|\mathfrak{m}}} \mu_K(\mathfrak{d}) = \begin{cases}
        1 & \text{ if $\mathfrak{m}$ is $h$-free}, \\
        0 & \text{ otherwise}.
    \end{cases}
\end{equation}
Moreover, the generating series for $\mu_K(\mathfrak{m})$ is the reciprocal of the Dedekind $\zeta$-function and is given by (for $\Re(s) > 1$)
\begin{equation}\label{dedekind}
    \frac{1}{\zeta_K(s)} = \sum_{\substack{\mathfrak{m} \in \mathcal{M} \\ \mathfrak{m} \neq 0}} \frac{\mu_K(\mathfrak{m})}{N(\mathfrak{m})^s} = \prod_{\wp \in \mathcal{P}} \left( 1 - \frac{1}{N(\wp)^s} \right).
\end{equation}
Using the above definitions, we prove:
\begin{lma}\label{hfreeidealrestrict}
    Let $h \geq 2$ be an integer. Let $\ell$ be a fixed prime ideal and $\mathcal{S}_{h,\ell}(x)$ denote the set of $h$-free ideals with norm $N(\cdot)$ less than or equal to $x$ and coprime to $\ell$. Let $R_{\mathcal{S}_h}(x)$ be defined in \eqref{RSh(x)}. Then, we have
    $$|\mathcal{S}_{h,\ell}(x)| = \left( \frac{N(\ell)^h - N(\ell)^{h-1}}{N(\ell)^h - 1} \right) \frac{\kappa}{\zeta_K(h)} x + O_h \big( R_{\mathcal{S}_h}(x) \big).$$
\end{lma}
\begin{proof}
    We define the modified M\"obius $\mu$-function on $\mathcal{M}$ as
    $$\mu_{K,\ell}(\mathfrak{m}) = \begin{cases}
        \mu_K(\mathfrak{m}) & \text{ if } (\ell,\mathfrak{m}) =1, \\
        0 & \text{ otherwise}.
    \end{cases}$$ 
    Using this and \eqref{iden_mu}, notice that
    \begin{align*}
        |\mathcal{S}_{h,\ell}(x)| & = \sum_{\substack{\mathfrak{m} \in \mathcal{M} \\ N(\mathfrak{m}) \leq x, (\mathfrak{m},\ell) = 1}} \sum_{\substack{\mathfrak{d} \in\mathcal{M} \\ \mathfrak{d}^h|m}} \mu_K(\mathfrak{d}) \\
        & = \sum_{\substack{\mathfrak{m} \in \mathcal{M} \\ N(\mathfrak{m}) \leq x, (\mathfrak{m},l) = 1}} \sum_{\substack{\mathfrak{d} \in\mathcal{M} \\ \mathfrak{d}^h|m}} \mu_{K,\ell}(\mathfrak{d}) \\
        & = \sum_{\substack{\mathfrak{d} \in\mathcal{M} \\ N(\mathfrak{d})^h \leq x}} \mu_{K,\ell}(\mathfrak{d}) \sum_{\substack{\mathfrak{m} \in \mathcal{M} \\ N(\mathfrak{m}) \leq x/N(\mathfrak{d})^h, (\mathfrak{m},l) = 1}} 1 \\
        & = \sum_{\substack{\mathfrak{d} \in\mathcal{M} \\ N(\mathfrak{d})^h \leq x}} \mu_{K,\ell}(\mathfrak{d}) \left( I_K(x/N(\mathfrak{d})^h)  - I_K(x/(N(\mathfrak{d})^h N(\ell))) \right).
    \end{align*}
    Using \eqref{Landaueq} and $N(\ell) > 1$, we obtain
    \begin{align*}
        |\mathcal{S}_{h,\ell}(x)| & = \kappa x \left( 1 - \frac{1}{N(\ell)} \right) \sum_{\substack{\mathfrak{d} \in\mathcal{M} \\ N(\mathfrak{d})^h \leq x}} \frac{\mu_{K,\ell}(\mathfrak{d})}{N(\mathfrak{d})^h} + O \left( \sum_{\substack{\mathfrak{d} \in\mathcal{M} \\ N(\mathfrak{d})^h \leq x}} \frac{x^{1 - \frac{2}{n_{\scaleto{K}{2.5pt}}+1}} }{N(\mathfrak{d})^{h(1 - \frac{2}{n_{\scaleto{K}{2.5pt}} + 1})}} \right) \\
        & = \kappa x \left( 1 - \frac{1}{N(\ell)} \right) \sum_{\substack{\mathfrak{d} \in\mathcal{M} \\ \mathfrak{d} \neq 0}} \frac{\mu_{K,l}(\mathfrak{d})}{N(\mathfrak{d})^h} + O \left( x \sum_{\substack{\mathfrak{d} \in\mathcal{M} \\ N(\mathfrak{d}) > x^{1/h}}} \frac{1}{N(\mathfrak{d})^h} \right) \\
        & \hspace{.5cm} + O \left( \sum_{\substack{\mathfrak{d} \in\mathcal{M} \\ N(\mathfrak{d}) \leq x^{1/h}}} \frac{x^{1 - \frac{2}{n_{\scaleto{K}{2.5pt}}+1}} }{N(\mathfrak{d})^{h(1 - \frac{2}{n_{\scaleto{K}{2.5pt}} + 1})}} \right).
    \end{align*}
Note that, using the Euler product formula mentioned in \eqref{dedekind}, we write the first sum on the right side above as
$$\sum_{\substack{\mathfrak{d} \in\mathcal{M} \\ \mathfrak{d} \neq 0}} \frac{\mu_{K,\ell}(\mathfrak{d})}{N(\mathfrak{d})^h} =  \sum_{\substack{\mathfrak{d} \in\mathcal{M} \\ \mathfrak{d} \neq 0, (\mathfrak{d},\ell) = 1}} \frac{\mu_K(\mathfrak{d})}{N(\mathfrak{d})^h} = \prod_{\substack{\wp \in \mathcal{P} \\ \wp \neq \ell}} \left( 1 - \frac{1}{N(\wp)^h} \right) = \frac{1}{\zeta_K(h)} \left( \frac{N(\ell)^h}{N(\ell)^h - 1} \right).$$
We bound the second sum using Part 7 of \lmaref{boundnm} with $\alpha = h$ as
\begin{equation*}
    \sum_{\substack{\mathfrak{d} \in\mathcal{M} \\ N(\mathfrak{d}) > x^{1/h}}} \frac{1}{N(\mathfrak{d})^h} \ll_h \frac{1}{x^{1 - \frac{1}{h}}}.
\end{equation*}
Also, for the third sum, using Parts 2, 3, and 5 of \lmaref{boundnm}, we obtain
\begin{equation*}
    \sum_{\substack{\mathfrak{d} \in\mathcal{M} \\ N(\mathfrak{d}) \leq x^{1/h}}} \frac{1}{N(\mathfrak{d})^{h(1 - \frac{2}{n_{\scaleto{K}{2.5pt}} + 1})}} \ll_h \begin{cases}
    1 & \text{ if } \frac{1}{h} < 1 - \frac{2}{n_{\scaleto{K}{3pt}} + 1}, \\
    \log x & \text{ if } \frac{1}{h} = 1 - \frac{2}{n_{\scaleto{K}{3pt}} + 1}, \text{ and} \\
    x^{\frac{1}{h} - 1 + \frac{2}{n_{\scaleto{K}{2.5pt}} + 1}} & \text{ if } \frac{1}{h} > 1 - \frac{2}{n_{\scaleto{K}{3pt}} + 1}.
\end{cases}
\end{equation*}
Combining the above four results, we complete the proof.
\end{proof}
Using the results discussed in this section, we prove the following Erd\H{o}s-Kac theorem over $h$-free ideals in the number field $K$. We achieve this by applying \thmref{yurugen} with $\mathcal{S} = \mathcal{S}_h$, $\beta = 1$, and $y = x^{1/\log \log x}$. We prove:
% \begin{thm}\label{erdoskachfreeideal}
% Let $x > 2$ be any real number and $h \geq 2$ be any integer. Let $\mathcal{S}_h(x)$ denote the set of $h$-free ideals with norm less than or equal to $x$. Then for $a \in \mathbb{R}$, we have
% $$\lim_{x \rightarrow \infty} \frac{1}{|\mathcal{S}_h(x)|} \left| \left\{ m \in \mathcal{S}_h(x) \ : \ \frac{\omega(m) - \log \log N(m)}{\sqrt{\log \log N(m)}} \leq a \right\} \right| = \Phi(a).$$
% \end{thm}
\begin{proof}[\textbf{Proof of \thmref{erdoskachfreeideal}}]
Consider the set $\mathcal{S} = \mathcal{S}_h$. By \eqref{hfreeidealcount}, we have
$$|\mathcal{S}(x)| = \frac{\kappa}{\zeta_K(h)} x + O_h \big( R_{\mathcal{S}_h}(x) \big),$$
where $R_{\mathcal{S}_h}(x)$ is defined in \eqref{RSh(x)}. Thus, 
$$\frac{|\mathcal{S}(x^{1/2})|}{|\mathcal{S}(x)|} \ll \frac{1}{x^{1/2}},$$
and hence the condition \eqref{Scondition}, $|\mathcal{S}(x^{1/2})| = o \left( |\mathcal{S}(x)| \right)$ is satisfied. For a fixed prime ideal $\wp$, let 
$$\mathcal{S}_\wp(x) := \left\{ \mathfrak{m} \in \mathcal{S}_h (x) \ : \ \wp | \mathfrak{m} \right\}.$$
%Let $p^k || n$ denote the property that $p^k |n$ but $p^{k+1} \nmid n$. 
Using \lmaref{hfreeidealrestrict} and \rmkref{remark1}, for large $x$, we obtain
\begin{align}\label{spx}
    |\mathcal{S}_\wp(x)| 
    % & = \sum_{\substack{n \leq x \\ n \in \mathcal{S}_h, p|n}} 1  
    %& = \sum_{k=1}^{h-1} \sum_{\substack{n \in \mathcal{S}_h(x) \\ p^k || n}} 1 \notag \\
    % & = \left( \sum_{k=1}^{h-1} \frac{1}{p^k} \right) \left( \frac{p^h - p^{h-1}}{p^h - 1} \right) \frac{x}{\zeta(h)} + O_h \left( x^{1/h}  \sum_{k=1}^{h-1}  \frac{1}{p^{k/h}} \right) \notag \\
    & = \sum_{k=1}^{h-1} |\mathcal{S}_{h,\wp} (x/N(\wp)^k)| \notag \\
    & = \sum_{k=1}^{h-1} \left( \left( \frac{N(\wp)^h - N(\wp)^{h-1})}{N(\wp)^h - 1} \right) \frac{1}{N(\wp)^k} \frac{\kappa}{\zeta_K(h)} x + O_h \left( R_{\mathcal{S}_h}(x/N(\wp)^k) \right) \right) \notag \\
    & = \frac{N(\wp)^{h-1} - 1}{N(\wp)^h - 1} \frac{\kappa}{\zeta_K(h)} x + O_h \left( \left( \frac{x}{N(\wp)} \right)^\tau \right),
\end{align}
where $\tau < 1$. Thus, we have
\begin{align*}
    \frac{|\mathcal{S}_\wp(x)|}{|\mathcal{S}(x)|} 
    % & = \frac{\frac{p^{h-1} - 1}{p^h - 1} \frac{x}{\zeta(h)} + O_h \left( \frac{x^{1/h}}{p^{1/h}} \right)}{\frac{x}{\zeta(h)} + O(x^{1/h})} \\
    % & = \frac{p^{h-1} - 1}{p^h - 1} + O_h \left( \frac{1}{x^{1 - \frac{1}{h}} p^{\frac{1}{h}}} \right) \\
    & = \lambda_\wp + e_\wp(x),
\end{align*}
where $\lambda_\wp = \frac{N(\wp)^{h-1} - 1}{N(\wp)^h - 1}$ and $e_\wp(x) = O_h \left( 1/ (x^{1 - \tau} N(\wp)^{\tau}) \right)$. 

Next, we choose $y = x^{1/\log \log x}$ and $\beta = 1$, and check that all the conditions in \thmref{yurugen} hold true. Note that the set in Condition (a) is empty and thus the condition holds trivially. By Part 4 of \lmaref{boundnm}, we obtain 
\begin{align*}
    \sum_{x^\frac{1}{\log \log x} < N(\wp) \leq x} \lambda_\wp \ll \sum_{x^\frac{1}{\log \log x} < N(\wp) \leq x} \frac{1}{N(\wp)} \ll \log \log \log x.
\end{align*}
Since $\log \log \log x = o \left( (\log \log x)^{1/2} \right)$, Condition (b) is satisfied as well. Using Part 1 of \lmaref{boundnm}, we have
$$\sum_{x^\frac{1}{\log \log x} < N(\wp) \leq x} |e_\wp(x)|  \ll_h \frac{1}{x^{1 - \tau}} \sum_{N(\wp) \leq x} \frac{1}{N(\wp)^{\tau}} \ll_h \frac{1}{\log x},$$
which makes Condition (c) true. Using 
$$ \frac{N(\wp)^{h-1} - 1}{N(\wp)^h - 1} = \frac{1}{N(\wp)} - \frac{N(\wp)-1}{N(\wp)(N(\wp)^h -1)}$$
with Part 3 and Part 4 of \lmaref{boundnm}, we obtain
\begin{align*}
\sum_{N(\wp) \leq x^\frac{1}{\log \log x}} \lambda_\wp & = \sum_{N(\wp) \leq x^\frac{1}{\log \log x}} \frac{1}{N(\wp)} + O(1) \\
& = \log \log x + O (\log \log \log x),
\end{align*}
which makes Condition (d) true. Finally, using Part 3 of \lmaref{boundnm} again, we have
$$\sum_{N(\wp) \leq x^\frac{1}{\log \log x}} \lambda_\wp^2 \ll \sum_{N(\wp) \leq x^\frac{1}{\log \log x}} \frac{1}{N(\wp)^2} \ll O(1).$$
This makes Condition (e) true. Thus, we are only required to verify Condition (f). Using \eqref{spx} and the Chinese Remainder Theorem, we obtain, for distinct prime ideals $\wp_{\scaleto{1}{3pt}},\cdots,\wp_{\scaleto{u}{3pt}}$, 
\begin{align*}
    & \left| \left\{ \mathfrak{m} \in \mathcal{S}(x) \ : \ \wp_i | \mathfrak{m} \text{ for all } i \in \{ 1, 2 ,\cdots, u\} \right\} \right| \\
    & = \left( \prod_{i=1}^u \frac{N(\wp_i)^{h-1} - 1}{N(\wp_i)^h - 1} \right) \frac{\kappa}{\zeta_K(h)} x + O_h \left( \frac{x^\tau}{\prod_{i=1}^u N(\wp_i)^\tau} \right).
\end{align*}
Thus
$$\frac{\left| \left\{ \mathfrak{m} \in \mathcal{S}(x) \ : \ \wp_i | \mathfrak{m} \text{ for all } i \in \{ 1, 2 ,\cdots, u\} \right\} \right|}{|\mathcal{S}(x)|} = \prod_{i=1}^u \frac{N(\wp_i)^{h-1} - 1}{N(\wp_i)^h - 1} + e_{\wp_{\scaleto{1}{3pt}} \cdots \wp_{\scaleto{u}{3pt}}}(x),$$
where
$$|e_{\wp_{\scaleto{1}{3pt}} \cdots \wp_{\scaleto{u}{3pt}}}(x)| \ll_h \frac{1}{x^{1 - \tau}}  \frac{1}{\prod_{i=1}^u N(\wp_i)^{\tau}}.$$
Let $r \in \mathbb{N}$. By the definition of $\sum{\vphantom{\sum}}''$ in the conditions mentioned before \thmref{yurugen}, and using Part 1 of \lmaref{boundnm} and $x^{1/\log \log x} = o(x^\epsilon)$ for any small $\epsilon > 0$, we have
\begin{align*}
    \sum{\vphantom{\sum}}'' |e_{\wp_{\scaleto{1}{3pt}} \cdots \wp_{\scaleto{u}{3pt}}}(x)| & \ll_h \frac{1}{x^{1 - \tau}}  \left( \sum_{N(\wp) \leq x^\frac{1}{\log \log x} } \frac{1}{N(\wp)^{\tau}} \right)^u \\
    & \ll_h \frac{1}{x^{1 - \tau}} \left( \frac{x^{\frac{1}{\log \log x} \left( 1 - \tau \right)}}{\log x} \right)^u \\
    & \ll_h \frac{1}{x^{1 - \tau - \epsilon'}},
\end{align*}
for any small $\epsilon' > 0$. Since, $ x^{- \left(1 - \tau - \epsilon'\right)} = o \left( (\log \log x)^{-r/2} \right)$, thus Condition (f) holds true as well. Since all the conditions of \thmref{yurugen} hold true with $\mathcal{S} = \mathcal{S}_h$, $\beta=1$ and $y = x^{1/\log \log x}$, thus applying \thmref{yurugen} completes the proof.
\end{proof}
\section{The Erd\H{o}s-Kac theorem over \texorpdfstring{$h$}{}-full ideals}
Let $h \geq 2$ be an integer. Recall that $\mathcal{N}_h$ denotes the set of $h$-full ideals in the number field $K$. Let $\mathfrak{m} \in \mathcal{N}_h$. We can separate the prime ideals in the prime ideal factorization of $\mathfrak{m}$ (see \eqref{factorization}) based on their powers modulo $h$ as 
\begin{equation}\label{hfullfactorize}
\mathfrak{m} = \underbrace{\wp_{1,0}^{h t_{1,0}} \cdots \wp_{l_0,0}^{h t_{l_0,0}}}_{\exp \equiv 0\pmod{h}} \underbrace{\wp_{1,1}^{h t_{1,1} + 1} \cdots \wp_{l_1,1}^{h t_{l_1,1}+1}}_{\exp \equiv 1\pmod{h}} \cdots \underbrace{\wp_{1,h-1}^{h t_{1,h-1} + h-1} \cdots \wp_{l_{h-1},h-1}^{h t_{l_{h-1},h-1} + h-1}}_{\exp \equiv h-1 \pmod{h}},
\end{equation}
where $\wp_{i,j}$ with $1 \leq i \leq l_j$ denote a distinct prime with its power congruent to $j$ modulo $h$ and $l_j$ denote the number of such distinct primes. 

We set
$$\mathfrak{a}_0 = \wp_{1,0}^{t_{1,0}} \cdots \wp_{l_0,0}^{t_{l_0,0}} \wp_{1,1}^{t_{1,1} - 1} \cdots \wp_{l_1,1}^{t_{l_1,1} - 1} \cdots \wp_{1,h-1}^{t_{1,h-1} - 1} \cdots \wp_{l_{h-1},h-1}^{ t_{l_{h-1},h-1} - 1}$$
and for $j = 1, \ldots, h-1$, we set
$$\mathfrak{a}_j = \wp_{1,j} \cdots \wp_{l_j,j}.$$
Note that $\mathfrak{a}_1,\cdots, \mathfrak{a}_{h-1}$ are square-free and co-prime to each other. 

Thus, we can write
\begin{equation}\label{defnhfull}
\mathfrak{m} = \mathfrak{a}_0^h \mathfrak{a}_1^{h+1} \cdots \mathfrak{a}_{h-1}^{2h -1}.
\end{equation}
Notice that the generating series for the $h$-full ideals is defined on $\Re(s) > 1/h$ as:
\begin{align}\label{formulanhk}
    \mathfrak{N}_h(s) & := \sum_{\mathfrak{m} \in \mathcal{N}_h} \frac{1}{N(\mathfrak{m})^s} \notag \\
    & = \prod_\wp \left( 1 + \frac{1}{N(\wp)^{hs}} + \cdots + \frac{1}{N(\wp)^{ks}} + \cdots \right) \notag \\
    & = \prod_\wp \left( 1 +  \frac{N(\wp)^{-hs}}{1 - N(\wp)^{-s}} \right).
\end{align}
Let $\mathcal{N}_h(x)$ denote the number of $h$-full ideals with a norm not exceeding $x$. Then, by the definition of an $h$-full ideal \eqref{defnhfull}, we have
$$\mathcal{N}_h(x) = \sum_{\substack{\mathfrak{m} \in \mathcal{N}_h \\ N(\mathfrak{m}) \leq x}} 1 = \sum_{\substack{\mathfrak{a}_0,\cdots, \mathfrak{a}_{h-1} \in \mathcal{M} \\ N(\mathfrak{a}_0^h \mathfrak{a}_1^{h+1} \cdots \mathfrak{a}_{h-1}^{2h -1}) \leq x}} \mu_K^2(\mathfrak{a}_1 \cdots \mathfrak{a}_{h-1}),$$
where the inner sum is non-zero if and only if $\mathfrak{a}_1,\cdots, \mathfrak{a}_{h-1}$ are square-free and co-prime to each other. By \cite[(1.5)]{is}, we have that there exists constants $\alpha_{r,h}$ where $2h+2 < r \leq (3h^2+h-2)/2$ such that the following equality is satisfied
\begin{equation}\label{iden2}
    \left( 1 + \frac{v^h}{1-v} \right) (1-v^h) (1 -v^{h+1}) \cdots (1-v^{2h-1}) = 1 - v^{2h+2} + \sum_{r=2h+3}^{(3h^2+h-2)/2} \alpha_{r,h} v^r.
\end{equation}
Note that $|\alpha_{r,h}| \leq h 2^h$. Now, substituting $v= N(\wp)^{-s}$ above, taking the product over all prime ideals, and using the product formula for the Dedekind $\zeta$-function \eqref{dedekind}, we obtain
\begin{equation}\label{Nhxformula}
\mathfrak{N}_h(s) = \zeta_K(hs) \zeta_K((h+1)s) \cdots \zeta_K((2h-1)s) \zeta_K^{-1}((2h+2)s) \phi_h^K(s),
\end{equation}
where $\phi_h^K(s)$ satisfies
\begin{equation}\label{iden1}
\prod_{\wp} \left( 1 - N(\wp)^{-(2h+2)s} + \sum_{r = 2h+3}^{(3h^2+h-2)/2} \alpha_{r,h} N(\wp)^{-rs} \right) = \zeta_K^{-1}((2h+2)s) \phi_h^K(s).
\end{equation}
For $h =2$, the sum on the right side of \eqref{iden2} is empty, and thus, $\phi_2^K(s) = 1$. Moreover, by \eqref{iden1}, if $h >2$, $\phi_h^K(s)$ has a Dirichlet series with abscissa of absolute convergence equal to $1/(2h+3)$. Thus, we may write
\begin{equation}\label{nhks}
    \mathfrak{N}_h(s) = \mathcal{G}_h(s) \mathcal{L}_h(s),
\end{equation}
where
\begin{equation}\label{L_h(s)}
\mathcal{L}_h(s) = \sum_{n=1}^\infty l_h(n) n^{-s} = \zeta_K(hs) \zeta_K((h+1)s) \cdots \zeta_K((2h-1)s),
\end{equation}
and
\begin{equation}\label{G_h(s)}
\mathcal{G}_h(s) = \sum_{n=1}^\infty g_h(n) n^{-s} = \frac{\phi_h^K(s)}{\zeta_K((2h+2)s)},
\end{equation}
is a Dirichlet series converging absolutely in $\Re(s) > 1/(2h+2)$. Thus, one can write
\begin{equation}\label{nhkx1}
    \mathcal{N}_h(x) = \sum_{\substack{\mathfrak{a}_0,\cdots, \mathfrak{a}_{h-1} \in \mathcal{M} \\ N(\mathfrak{a}_0^h \mathfrak{a}_1^{h+1} \cdots \mathfrak{a}_{h-1}^{2h -1}) \leq x}} \mu_K^2(\mathfrak{a}_1 \cdots \mathfrak{a}_{h-1}) = \sum_{m n \leq x} g_h(m) l_h(n).
\end{equation}
Additionally, let $\mathcal{T}_h(x)$ denote the unweighted sum
\begin{equation}\label{T_h(x)}
\mathcal{T}_h(x) := \sum_{n \leq x} l_h(n) = \sum_{N(\mathfrak{a}_0^h \mathfrak{a}_1^{h+1} \cdots \mathfrak{a}_{h-1}^{2h -1}) \leq x} 1,
\end{equation}
where the sum is taken over all ideals $\mathfrak{a}_0, \ldots, \mathfrak{a}_{h-1}$. Clearly, $\mathcal{N}_h(x) \leq \mathcal{T}_h(x)$. Thus, to prove an estimate for $\mathcal{N}_h(x)$, we start by proving an estimate for $\mathcal{T}_h(x)$. We prove:
\begin{lma}\label{T_h(X)bound}
For any $x > 1$ and any integer $h \geq 2$, we have
$$\mathcal{T}_h(x) = \kappa \left( \prod_{i=1}^{h-1} \zeta_K \left( 1+i/h \right) \right) x^{1/h} + O_h \big( R_{\mathcal{T}_h}(x) \big),$$
where
\begin{equation}\label{E1(x)}
    R_{\mathcal{T}_h}(x) = \begin{cases}
        x^{\frac{1}{h} \left( 1 - \frac{2}{n_K+1}\right)} & \text{ if } \frac{h}{h+1} < 1 - \frac{2}{n_{\scaleto{K}{3pt}} + 1}, \\
        x^{\frac{1}{h} \left( 1 - \frac{2}{n_K+1}\right)} (\log x) & \text{ if } \frac{h}{h+1} = 1 - \frac{2}{n_{\scaleto{K}{3pt}} + 1}, \text{ and } \\
        x^{1/(h+1)} & \text{ if } \frac{h}{h+1} > 1 - \frac{2}{n_{\scaleto{K}{3pt}} + 1}.
    \end{cases}
    \end{equation}
\end{lma}
\begin{proof}
Using \eqref{Landaueq} and $\upsilon = 1 - 2/(n_K+1) < 1$, we have
\begin{align}\label{Teq1}
    \mathcal{T}_h(x) & = \sum_{N(\mathfrak{a}_1^{h+1} \cdots \mathfrak{a}_{h-1}^{2h -1}) \leq x} \sum_{N(\mathfrak{a}_0^h) \leq \frac{x}{N(\mathfrak{a}_1^{h+1} \cdots \mathfrak{a}_{h-1}^{2h -1})}} 1 \notag\\
    & = \sum_{N(\mathfrak{a}_1^{h+1} \cdots \mathfrak{a}_{h-1}^{2h -1}) \leq x} \left( \frac{\kappa}{(N(\mathfrak{a}_1^{h+1} \cdots \mathfrak{a}_{h-1}^{2h -1}))^{1/h}} x^{1/h} + O \left( \left( \frac{x^{1/h}}{(N(\mathfrak{a}_1^{h+1} \cdots \mathfrak{a}_{h-1}^{2h -1}))^{1/h}} \right)^\upsilon \right)\right) \notag\\
    & = \kappa x^{1/h} \mathfrak{T}_1 + O(\mathfrak{T}_2),
\end{align}  

where $\mathfrak{T}_1$ and $\mathfrak{T}_2$ are defined as
$$\mathfrak{T}_1 = \sum_{N(\mathfrak{a}_1^{h+1} \cdots \mathfrak{a}_{h-1}^{2h -1}) \leq x} \frac{1}{(N(\mathfrak{a}_1^{h+1} \cdots \mathfrak{a}_{h-1}^{2h -1}))^{1/h}},$$
and
$$\mathfrak{T}_2 = x^{\upsilon/h} \sum_{N(\mathfrak{a}_1^{h+1} \cdots \mathfrak{a}_{h-1}^{2h -1}) \leq x} \frac{1}{(N(\mathfrak{a}_1^{h+1} \cdots \mathfrak{a}_{h-1}^{2h -1}))^{\upsilon/h}}.$$
Notice that, by Part 7 of \lmaref{boundnm}, we have
\begin{align*}
    \mathfrak{T}_1 & = \sum_{N(\mathfrak{a}_2^{h+2} \cdots \mathfrak{a}_{h-1}^{2h -1}) \leq x} \frac{1}{(N(\mathfrak{a}_2^{h+2} \cdots \mathfrak{a}_{h-1}^{2h -1}))^{1/h}} \sum_{N ( \mathfrak{a}_1^{h+1}) \leq \frac{x}{N \left(\mathfrak{a}_2^{h+2} \cdots \mathfrak{a}_{h-1}^{2h -1} \right)}} \frac{1}{\mathfrak{a}_1^{1+1/h}} \\
    & = \sum_{N(\mathfrak{a}_2^{h+2} \cdots \mathfrak{a}_{h-1}^{2h -1}) \leq x} \frac{1}{(N(\mathfrak{a}_2^{h+2} \cdots \mathfrak{a}_{h-1}^{2h -1}))^{1/h}} \bigg( \zeta_K(1+1/h) \\
    & \hspace{.5cm} + O \bigg( \sum_{N(\mathfrak{a}_1) > \frac{x^{1/(h+1)}}{\left(N \left(\mathfrak{a}_2^{h+2} \cdots \mathfrak{a}_{h-1}^{2h -1} \right) \right)^{1/(h+1)}}} \frac{1}{\mathfrak{a}_1^{1+1/h}} \bigg) \bigg) \\
    & = \zeta_K(1+1/h) \sum_{N(\mathfrak{a}_2^{h+2} \cdots \mathfrak{a}_{h-1}^{2h -1}) \leq x} \frac{1}{(N(\mathfrak{a}_2^{h+2} \cdots \mathfrak{a}_{h-1}^{2h -1}))^{1/h}} \\
    & \hspace{.5cm} + O \left( \frac{1}{x^{1/(h(h+1))}}  \sum_{N(\mathfrak{a}_2^{h+2} \cdots \mathfrak{a}_{h-1}^{2h -1}) \leq x} \frac{1}{(N(\mathfrak{a}_2^{h+2} \cdots \mathfrak{a}_{h-1}^{2h -1}))^{1/(h+1)}} \right).
\end{align*}
Note that, for any $h \geq 2$, the sum inside the error term above is $O_h(1)$. Thus, we deduce
$$\mathfrak{T}_1 =  \zeta_K(1+1/h) \sum_{N(\mathfrak{a}_2^{h+2} \cdots \mathfrak{a}_{h-1}^{2h -1}) \leq x} \frac{1}{(N(\mathfrak{a}_2^{h+2} \cdots \mathfrak{a}_{h-1}^{2h -1}))^{1/h}} +  O_h \left( \frac{1}{x^{1/(h(h+1))}} \right).$$
Repeating this process $h-1$ times and realizing that $x^{-i/(h(h+i))} \leq x^{-1/(h(h+1))}$ for all $i \geq 1$, we obtain
\begin{align*}
    \mathfrak{T}_1 & = \prod_{i=1}^{2} \zeta_K(1+i/h) \sum_{N(\mathfrak{a}_3^{h+3} \cdots \mathfrak{a}_{h-1}^{2h -1}) \leq x} \frac{1}{(N(\mathfrak{a}_3^{h+3} \cdots \mathfrak{a}_{h-1}^{2h -1}))^{1/h}} +  O_h \left( \frac{1}{x^{1/(h(h+1))}} \right) \notag \\
    & \vdots \notag \\
    & = \prod_{i=1}^{h-2} \zeta_K(1+i/h) \sum_{N(\mathfrak{a}_{h-1}^{2h -1}) \leq x} \frac{1}{(N(\mathfrak{a}_{h-1}^{2h -1}))^{1/h}} +  O_h \left( \frac{1}{x^{1/(h(h+1))}} \right) \notag \\
    & = \prod_{i=1}^{h-1} \zeta_K(1+i/h) + O_h \left( \frac{1}{x^{1/(h(h+1))}} \right).
\end{align*}
Thus,
\begin{equation*}
\kappa x^{1/h} \mathfrak{T}_1 = \kappa \left( \prod_{i=1}^{h-1} \zeta_K \left( 1+i/h \right) \right) x^{1/h} + O_h \left( x^{1/(h+1)} \right).   
\end{equation*}
Next, we bound the term $\mathfrak{T}_2$. Note that, by Parts 2, 3, and 5 of \lmaref{boundnm}, we have
\begin{equation*}
    \mathfrak{T}_2 \ll x^{\upsilon/h} \sum_{N(\mathfrak{a}_{\scaleto{1}{3pt}}) \leq x^{1/(h+1)}} \frac{1}{\mathfrak{a}_1^{\upsilon(1+1/h)}} \ll_h \begin{cases}
        x^{\upsilon/h} & \text{ if } \upsilon(1+1/h) > 1, \\
        x^{\upsilon/h} (\log x) & \text{ if } \upsilon(1+1/h) = 1, \text{ and } \\
        x^{1/(h+1)} & \text{ if } \upsilon(1+1/h) < 1.
    \end{cases}
\end{equation*}
Combining the last two results with \eqref{Teq1} completes the proof.
\end{proof}
Using the last lemma, we establish the distribution result for $h$-full ideals. We prove:
\begin{proof}[\textbf{Proof of \thmref{hfullideals}}]
Using \eqref{nhkx1}, \eqref{T_h(x)} and \lmaref{T_h(X)bound}, we have
\begin{align}\label{hfulleqn1}
    \mathcal{N}_h(x) & = \sum_{m \leq x} g_h(m) \sum_{n \leq x/m} l_h(n) \notag \\
    & = \sum_{m \leq x} g_h(m) \mathcal{T}_h(x/m) \notag \\
    & = \kappa \left( \prod_{i=1}^{h-1} \zeta_K \left( 1+i/h \right) \right) x^{1/h} \sum_{m \leq x} \frac{g_h(m)}{m^{1/h}} + O \left( \sum_{m \leq x} g_h(m) R_{\mathcal{T}_h}(x/m) \right).
\end{align}
Recall that 
$$\mathcal{G}_h(s) = \sum_{n=1}^\infty g_h(n) n^{-s} = \frac{\phi_h^K(s)}{\zeta_K((2h+2)s)}$$
converges absolutely in $\Re(s) > 1/(2h+2)$. Therefore, for any $\epsilon > 0$,
$$\sum_{n \leq x} g_h(n) \ll x^{(1/(2h+2))+ \epsilon},$$
and thus, by partial summation,
\begin{equation}\label{hfulleqn2}
    \sum_{m \leq x} \frac{g_h(m)}{m^{1/h}} = \mathcal{G}_h (1/h) + O \left( \sum_{m > x} \frac{g_h(m)}{m^{1/h}} \right) = \mathcal{G}_h(1/h) + O(x^{(1/(2h+2)) + \epsilon - 1/h}).
\end{equation}
Moreover, by the definition of $R_{\mathcal{T}_h}(x)$ \eqref{E1(x)}, we have
\begin{equation*}
    \sum_{m \leq x} g_h(m) R_{\mathcal{T}_h}(x/m) = \begin{cases}
        x^{\frac{1}{h} \left( 1 - \frac{2}{n_{\scaleto{K}{2.5pt}}+1}\right)} \sum_{m \leq x} \frac{g_h(m)}{m^{(1 - 2/(n_{\scaleto{K}{2.5pt}}+1))/h}} & \text{ if } \frac{h}{h+1} < 1 - \frac{2}{n_{\scaleto{K}{3pt}} + 1}, \\
        x^{\frac{1}{h+1}} \sum_{m \leq x} \frac{g_h(m)}{m^{1/(h+1)}} (\log (x/m)) & \text{ if } \frac{h}{h+1} = 1 - \frac{2}{n_{\scaleto{K}{3pt}} + 1}, \\
        x^{\frac{1}{h+1}} \sum_{m \leq x} \frac{g_h(m)}{m^{1/(h+1)}} & \text{ if } \frac{h}{h+1} > 1 - \frac{2}{n_{\scaleto{K}{3pt}} + 1}.
    \end{cases}
    \end{equation*}
Again, since $\mathcal{G}_h(s)$ converges at $1/(h+1)$, for the third case above, we have
$$x^{\frac{1}{h+1}} \sum_{m \leq x} \frac{g_h(m)}{m^{1/(h+1)}} \ll_h x^{\frac{1}{h+1}},$$
and for the middle case, we have
\begin{equation*}
    x^{\frac{1}{h+1}} \sum_{m \leq x} \frac{g_h(m)}{m^{1/(h+1)}} (\log (x/m)) \ll_h x^{\frac{1}{h+1}} (\log x).
\end{equation*}
Moreover, for the first case, since $(1 - (2/(n_{\scaleto{K}{4pt}} + 1))/h > 1/(h+1)$, thus
$$x^{\frac{1}{h} \left( 1 - \frac{2}{n_{\scaleto{K}{3pt}}+1}\right)} \sum_{m \leq x} \frac{g_h(m)}{m^{(1 - 2/(n_{\scaleto{K}{3pt}}+1))/h}} \ll x^{\frac{1}{h} \left( 1 - \frac{2}{n_{\scaleto{K}{3pt}}+1}\right)} \sum_{m \leq x} \frac{g_h(m)}{m^{1/(h+1)}} \ll_h x^{\frac{1}{h} \left( 1 - \frac{2}{n_{\scaleto{K}{3pt}}+1}\right)}.$$
Putting the last four results together, we have
\begin{equation}\label{sumwithE_1}
    \sum_{m \leq x} g_h(m) R_{\mathcal{T}_h}(x/m) = \begin{cases}
        x^{\frac{1}{h} \left( 1 - \frac{2}{n_{\scaleto{K}{2.5pt}}+1}\right)} & \text{ if } \frac{h}{h+1} < 1 - \frac{2}{n_{\scaleto{K}{3pt}} + 1}, \\
        x^{\frac{1}{h+1}} (\log x) & \text{ if } \frac{h}{h+1} = 1 - \frac{2}{n_{\scaleto{K}{3pt}} + 1}, \text{ and}\\
        x^{\frac{1}{h+1}} & \text{ if } \frac{h}{h+1} > 1 - \frac{2}{n_{\scaleto{K}{3pt}} + 1}.
    \end{cases} 
\end{equation}
Combining the above together with \eqref{hfulleqn1} and \eqref{hfulleqn2}, we obtain
\begin{equation*}
    \mathcal{N}_h(x) = \kappa \left( \prod_{i=1}^{h-1} \zeta_K \left( 1+i/h \right) \right) \mathcal{G}_h(1/h) x^{1/h} + O \big( R_{\mathcal{N}_h}(x) \big),
\end{equation*}
where $R_{\mathcal{N}_h}(x)$ is defined in \eqref{E2(x)}. Finally, by \eqref{Nhxformula}, \eqref{formulanhk} and \eqref{dedekind}, we have
\begin{align*}
\left( \prod_{i=1}^{h-1} \zeta_K \left( 1+i/h \right) \right) \mathcal{G}_h(1/h) & = \prod_\wp \left( 1 +  \frac{N(\wp)^{-1}}{1 - N(\wp)^{-1/h}} \right) \left( 1 - \frac{1}{N(\wp)} \right) \\
& = \prod_\wp \left( 1 + \frac{N(\wp) - N(\wp)^{1/h}}{N(\wp)^2 \left( N(\wp)^{1/h} - 1 \right)}\right),
\end{align*}
which is a convergent product, since the left-hand side is a finite value. This completes the proof.
\end{proof}
We use the technique in the previous lemma to establish the distribution of $h$-full ideals which are co-prime to a given prime ideal. We prove:
\begin{lma}\label{hfullidealsrestrict}
Let $h \geq 2$ be an integer. Let $\ell$ be a fixed prime ideal and $\mathcal{N}_{h,\ell}(x)$ denote the set of $h$-full ideals with norm $N(\cdot)$ less than or equal to $x$ and coprime to $\ell$. Let $\gamma_{\scaleto{h}{4.5pt}}$ be defined in \eqref{gammahk} and $R_{\mathcal{N}_h}(x)$ be defined in \eqref{E2(x)}. Then, we have
$$|\mathcal{N}_{h,\ell}(x)| = \frac{\kappa \gamma_{\scaleto{h}{4.5pt}}}{\left( 1+ \frac{N(\ell)^{-1}}{1-N(\ell)^{-1/h}} \right)} x^{1/h} + O_h \big( R_{\mathcal{N}_h}(x) \big).$$
\end{lma}
\begin{proof}
Similar to \eqref{formulanhk}, we have
\begin{equation}\label{nhlks}
   \mathfrak{N}_{h,\ell}(s) := \sum_{\substack{\mathfrak{m} \in \mathcal{N}_h \\ (\mathfrak{m},\ell)=1}} \frac{1}{N(\mathfrak{m})^s} = \prod_{\substack{\wp \\ \wp \neq \ell}} \left( 1 +  \frac{N(\wp)^{-hs}}{1 - N(\wp)^{-s}} \right). 
\end{equation}
Moreover, by \eqref{nhks}, we can write
$$\mathfrak{N}_{h,\ell}(s) = \mathcal{G}_{h,\ell}(s) \mathcal{L}_h(s),$$
where $\mathcal{L}_h(s)$ is defined in \eqref{L_h(s)} and
\begin{equation}\label{Ghkl(s)}
\mathcal{G}_{h,\ell}(s) = \sum_{n \leq x} g_{h,\ell}(n) n^{-s}= \frac{\phi_h^K(s)}{ \left( 1 +  \frac{N(\ell)^{-hs}}{1 - N(\ell)^{-s}} \right)\zeta_K((2h+2)s)},
\end{equation}
where $\phi_k^K(s)$ is defined in \eqref{Nhxformula}. Since, $\phi_h^K(s)/ \zeta_K((2h+2)s)$ converges absolutely in $\Re(s) > 1/(2h+2)$, we can deduce that $\mathcal{G}_{h,\ell}(s)$ converges absolutely in $\Re(s) > 1/(2h+2)$ as well (see \cite[Page 14]{dkl2}). Thus, \eqref{T_h(x)} and \lmaref{T_h(X)bound}, we have
\begin{align*}\label{nhlkx1}
    \mathcal{N}_{h,\ell}(x) & = \sum_{m n \leq x} g_{h,\ell}(m) l_h(n) \\
    & = \sum_{m \leq x} g_{h,\ell}(m) \mathcal{T}_h(x/m) \\
    & = \kappa \left( \prod_{i=1}^{h-1} \zeta_K \left( 1+i/h \right) \right) x^{1/h} \sum_{m \leq x} \frac{g_{h,\ell}(m)}{m^{1/h}} + O \left( \sum_{m \leq x} g_{h,\ell}(m) R_{\mathcal{T}_h}(x/m) \right).
\end{align*}
Since $\mathcal{G}_{h,\ell}(s)$ converges absolutely in $\Re(s) > 1/(2h+2)$, thus, for any $\epsilon > 0$,
$$\sum_{n \leq x} g_{h,\ell}(n) \ll x^{(1/(2h+2))+ \epsilon},$$
and thus, by partial summation,
\begin{equation}\label{hfulleqn2l}
    \sum_{m \leq x} \frac{g_{h,\ell}(m)}{m^{1/h}} = \mathcal{G}_{h,\ell}(1/h) + O \left( \sum_{m > x} \frac{g_{h,\ell}(m)}{m^{1/h}} \right) = \mathcal{G}_{h,\ell}(1/h) + O(x^{(1/(2h+2)) + \epsilon - 1/h}).
\end{equation}
Moreover, working similar to \eqref{sumwithE_1}, we obtain
\begin{equation*}\label{sumwithE_1l}
    \sum_{m \leq x} g_{h,\ell}(m) R_{\mathcal{T}_h}(x/m) = O \big( R_{\mathcal{N}_h}(x) \big),
\end{equation*}
where $R_{\mathcal{N}_h}(x)$ is defined in \eqref{E2(x)}. Combining the last four results, we obtain
\begin{equation*}
    \mathcal{N}_{h,\ell}(x) = \kappa \left( \prod_{i=1}^{h-1} \zeta_K \left( 1+i/h \right) \right) \mathcal{G}_{h,\ell}(1/h) x^{1/h} + O \big( R_{\mathcal{N}_h}(x) \big).
\end{equation*}
Finally, by \eqref{nhlks}, \eqref{Ghkl(s)}, \eqref{Nhxformula}, and \eqref{dedekind}, we obtain
\begin{align*}
\left( \prod_{i=1}^{h-1} \zeta_K \left( 1+i/h \right) \right) \mathcal{G}_{h,\ell}(1/h) & = \prod_{\substack{\wp \\ \wp \neq \ell}}  \left( 1 +  \frac{N(\wp)^{-1}}{1 - N(\wp)^{-1/h}} \right) \prod_\wp \left( 1 - \frac{1}{N(\wp)} \right) \\
& = \frac{\gamma_{\scaleto{h}{4.5pt}}}{\left( 1+ \frac{N(\ell)^{-1}}{1-N(\ell)^{-1/h}} \right)},
\end{align*}
where $\gamma_{\scaleto{h}{4.5pt}}$ is defined in \eqref{gammahk}. This completes the proof.
\end{proof}
Using the results discussed in this subsection, we prove the following Erd\H{o}s-Kac theorem over $h$-full ideals in the number field $K$. We achieve this by applying \thmref{yurugen} with $\mathcal{S} = \mathcal{N}_h$, $\beta = 1/h$, and $y = x^{1/(h \log \log x)}$. We prove:
% \begin{thm}\label{erdoskachfullideal}
% Let $x > 2$ be any real number and $h \geq 2$ be any integer. Let $\mathcal{N}_h(x)$ denote the set of $h$-full ideals with norm $N(\cdot)$ less than or equal to $x$. Then for $a \in \mathbb{R}$, we have
% $$\lim_{x \rightarrow \infty} \frac{1}{|\mathcal{N}_h(x)|} \left| \left\{ \mathfrak{m} \in \mathcal{N}_h(x) \ : \ \frac{\omega(\mathfrak{m}) - \log \log N(\mathfrak{m})}{\sqrt{\log \log N(\mathfrak{m})}} \leq a \right\} \right| = \Phi(a).$$
% \end{thm}
\begin{proof}[\textbf{Proof of \thmref{erdoskachfullideal}}]
Consider the set $\mathcal{S} = \mathcal{N}_h$. By \thmref{hfullideals}, we have
$$|\mathcal{S}(x)| = \kappa \gamma_{\scaleto{h}{4.5pt}} x^{1/h} + O_h \big( R_{\mathcal{N}_h}(x) \big),$$
where $R_{\mathcal{N}_h}(x) \ll x^{\xi/h}$ for some $0 < \xi < 1$.
Therefore, 
$$\frac{|\mathcal{S}(x^{1/2})|}{|\mathcal{S}(x)|} \ll \frac{1}{x^{1/(2h)}},$$
and hence, $|\mathcal{S}(x^{1/2})| = o \left( |\mathcal{S}(x)| \right)$ is satisfied. For a fixed prime ideal $\wp$, let 
$$\mathcal{S}_\wp(x) := \left\{ \mathfrak{m} \in \mathcal{N}_h(x) \ : \ \wp | \mathfrak{m} \right\}.$$
Using \lmaref{hfullidealsrestrict} with $R_{\mathcal{N}_h}(x) \ll x^{\xi/h}$, we obtain
\begin{align}\label{npx}
    |\mathcal{S}_\wp(x)| 
    % & = \sum_{\substack{n \leq x \\ n \in \mathcal{N}_h, p|n}} 1  \notag \\
    % & = \sum_{k=h}^{\lfloor \frac{\log x}{\log p} \rfloor} \sum_{\substack{n \in \mathcal{N}_h(x) \\ p^k || n}} 1 \notag\\
    & = \sum_{k=h}^{\left\lfloor \frac{\log x}{\log N(\wp)} \right\rfloor} \left| \mathcal{N}_{h,\wp}(x/N(\wp)^k) \right|  \notag\\
    & = \left( \sum_{k=h}^{\infty} \frac{1}{N(\wp)^{k/h}} \right)\frac{ \kappa \gamma_{\scaleto{h}{4.5pt}}}{\left( 1+ \frac{N(\wp)^{-1}}{1-N(\wp)^{-1/h}} \right)} x^{1/h} + O_h \left( x^{\xi/h}   \sum_{k=h}^{\left\lfloor \frac{\log x}{\log N(\wp)} \right\rfloor}  \frac{1}{N(\wp)^{k\xi/h}} \right) \notag \\
    & = \frac{\kappa \gamma_{\scaleto{h}{4.5pt}}}{N(\wp) (1-N(\wp)^{-1/h}+N(\wp)^{-1})} x^{1/h} + O_h \left( \frac{x^{\xi/h}}{N(\wp)^{\xi}} \right).
\end{align}
Thus,
\begin{align*}
    \frac{|\mathcal{S}_\wp(x)|}{|\mathcal{S}(x)|} 
    % & = \frac{\frac{\gamma_{0,h}}{p (1-p^{-1/h}+p^{-1})} x^{1/h} + O_h \left( \frac{x^{1/(h+1)}}{p^{h/(h+1)}} \right)}{\gamma_{0,h} x^{1/h} + O_h(x^{1/(h+1)})} \\
    % & = \frac{1}{p (1-p^{-1/h}+p^{-1})} + O_h \left( \frac{1}{x^{\frac{1}{h(h+1)}} p^{\frac{h}{h+1}}} \right) \\
    & = \lambda_\wp + e_\wp(x),
\end{align*}
where $\lambda_\wp = \frac{1}{N(\wp) (1-N(\wp)^{-1/h}+N(\wp)^{-1})}$ and $e_\wp(x) = O_h \left( \frac{1}{x^{(1 - \xi)/h} N(\wp)^{\xi}} \right)$. 

Next, we choose $y = x^\frac{1}{h \log \log x}$, and $\beta = 1/h$, and check again that all the conditions in \thmref{yurugen} hold true. Note that the set in Condition (a) is empty and thus the condition holds trivially. By Part 4 of \lmaref{boundnm}, we obtain 
\begin{align*}
    \sum_{x^\frac{1}{h \log \log x} < N(\wp) \leq x^{1/h}} \lambda_\wp \ll \sum_{x^\frac{1}{h \log \log x} < N(\wp) \leq x^{1/h}} \frac{1}{N(\wp)} \ll \log \log \log x,
\end{align*}
which makes Condition (b) true. Using Part 1 of \lmaref{boundnm}, we have
$$\sum_{x^\frac{1}{h\log \log x} < N(\wp) \leq x^{1/h}} |e_\wp(x)|  \ll_h \frac{1}{x^{(1-\xi)/h}} \sum_{N(\wp) \leq x^{1/h}} \frac{1}{N(\wp)^{\xi}} \ll_h \frac{1}{\log x},$$
which makes Condition (c) true. Using 
$$ \frac{1}{N(\wp)(1- N(\wp)^{-1/h} + N(\wp)^{-1})} = \frac{1}{N(\wp)} + \frac{N(\wp)^{-1/h}-N(\wp)^{-1}}{N(\wp)(1 - N(\wp)^{-1/h} + N(\wp)^{-1})},$$
and Parts 3 and 4 of \lmaref{boundnm} again, we obtain
\begin{align*}
\sum_{N(\wp) \leq x^\frac{1}{h \log \log x}} \lambda_\wp & = \sum_{N(\wp) \leq x^\frac{1}{h \log \log x}} \frac{1}{N(\wp)} + O(1) \\
& = \log \log x + O (\log \log \log x),
\end{align*}
which makes Condition (d) true. Finally, again using Part 3 of \lmaref{boundnm} with $\alpha =2$, we have
$$\sum_{N(\wp) \leq x^\frac{1}{h \log \log x}} \lambda_\wp^2 \ll \sum_{N(\wp) \leq x^\frac{1}{h \log \log x}} \frac{1}{N(\wp)^2} \ll O(1).$$
This makes Condition (e) true. Thus, similar to the case of $h$-free ideals, we are only required to verify Condition (f). Using \eqref{npx} and the Chinese Remainder Theorem, we obtain, for distinct prime ideals $\wp_{\scaleto{1}{3pt}},\cdots,\wp_{\scaleto{u}{3pt}}$, 
\begin{align*}
    & \left| \left\{ \mathfrak{m} \in \mathcal{S}(x) \ : \ \wp_i | \mathfrak{m} \text{ for all } i \in \{ 1, 2 ,\cdots, u\} \right\} \right| \\
    & = \left( \prod_{i=1}^u \frac{1}{N(\wp_i) (1-N(\wp_i)^{-1/h}+N(\wp_i)^{-1})}  \right) \kappa \gamma_{\scaleto{h}{4.5pt}} x^{1/h} + O_h \left( \frac{x^{\xi/h}}{\prod_{i=1}^u N(\wp_i)^{\xi}} \right).
\end{align*}
Thus
\begin{align*}
    & \frac{\left| \left\{ \mathfrak{m} \in \mathcal{S}(x) \ : \ \wp_i | \mathfrak{m} \text{ for all } i \in \{ 1, 2 ,\cdots, u\} \right\} \right|}{|\mathcal{S}(x)|} \\
    & = \prod_{i=1}^u \frac{1}{N(\wp_i) (1-N(\wp_i)^{-1/h}+N(\wp_i)^{-1})} + e_{\wp_{\scaleto{1}{3pt}} \cdots \wp_{\scaleto{u}{3pt}}}(x),
\end{align*}
where
$$|e_{\wp_{\scaleto{1}{3pt}} \cdots \wp_{\scaleto{u}{3pt}}}(x)| \ll_h \frac{1}{x^{(1 - \xi)/h}}  \frac{1}{\prod_{i=1}^u N(\wp_i)^{\xi}}.$$
Let $r \in \mathbb{N}$. By the definition of $\sum{\vphantom{\sum}}''$ in the conditions mentioned before \thmref{yurugen}, and using Part 1 of \lmaref{boundnm} and $x^{(1-\xi)/(h \log \log x)} = o(x^\epsilon)$ for any small $\epsilon > 0$, we have
\begin{align*}
    \sum{\vphantom{\sum}}'' |e_{\wp_{\scaleto{1}{3pt}} \cdots \wp_{\scaleto{u}{3pt}}}(x)| & \ll_h \frac{1}{x^{(1 - \xi)/h}}  \left( \sum_{N(\wp) \leq x^\frac{1}{h \log \log x} } \frac{1}{N(\wp)^{\xi}} \right)^u 
    % & \ll_h \frac{1}{x^{\frac{h}{h+1}}} \left( \frac{x^{\frac{1}{h^2 (h+1) \log \log x}}}{\log x} \right)^u \\
    \ll_h \frac{1}{x^{((1-\xi)/h) - \epsilon'}},
\end{align*}
for any small $\epsilon' > 0$. Since, $x^{-(((1-\xi)/h) - \epsilon')} = o \left( (\log \log x)^{-r/2} \right)$, thus Condition (f) holds true as well. Since all the conditions of \thmref{yurugen} hold true with $\mathcal{S} = \mathcal{N}_h$, $\beta=1/h$ and $y = x^{1/(h \log \log x)}$, thus applying \thmref{yurugen} completes the proof.
\end{proof}

In this work, we establish that $\omega(\mathfrak{m})$ satisfies the Gaussian distribution over any subset of the set of ideals in a number field satisfying certain additional conditions. In particular, we establish that the set of $h$-free ideals and the set of $h$-full ideals satisfy these conditions and thus $\omega(\mathfrak{m})$ satisfies the Gaussian distribution over these subsets. In addition, we can generalize the Erd\H{o}s-Kac theorem to any subset of any abelian monoid that satisfies certain additional conditions similar to the number field case. We will report this result in a follow-up article.

\section{Acknowledgement}

We note that Section 4 of this manuscript is based on the work of the third author carried out during her Ph.D. studies under the supervision of Barry C. Mazur. Therefore, the authors would like to thank Prof. Mazur for his guidance and motivation.

\bibliographystyle{plain} % We choose the "plain" reference style
\bibliography{mybib.bib} % Entries are in the biblio.bib file

\end{document}